\documentclass[11pt,english,a4paper]{amsart} 
\usepackage[T1]{fontenc}
\usepackage[latin1]{inputenc}
\usepackage[a4paper]{geometry}
\usepackage{mathrsfs}
\usepackage{amsthm}
\usepackage{amsmath}
\usepackage{amssymb}
\usepackage{graphicx}
\usepackage{caption}
\usepackage{subcaption}
\usepackage{color}
\usepackage{a4wide}
\usepackage{amscd}
\usepackage{latexsym}
\usepackage{dsfont}
\usepackage{etoolbox} 

\makeatletter
  \theoremstyle{remark}
  \newtheorem*{acknowledgement*}{\protect\acknowledgementname}
  
  \theoremstyle{plain}
  \newtheorem{theorem}{\protect\theoremname}[section]

  \theoremstyle{plain}

  \theoremstyle{plain}
  \newtheorem{definition}[theorem]{\protect\definitionname}
  
  \theoremstyle{plain}

  \theoremstyle{plain}

  \theoremstyle{plain}
  \newtheorem{lemma}[theorem]{\protect\lemmaname}
    
  \theoremstyle{plain}
  \newtheorem{proposition}[theorem]{\protect Proposition}

  \theoremstyle{plain}

  \theoremstyle{plain}
  \newtheorem{corollary}[theorem]{\protect Corollary}

\usepackage{MnSymbol}
\usepackage{fancyhdr}
\usepackage{esint}

\makeatother

\usepackage{babel}
  \providecommand{\acknowledgementname}{Acknowledgement}
  \providecommand{\definitionname}{Definition}
  \providecommand{\examplename}{Example}
  \providecommand{\lemmaname}{Lemma}
  \providecommand{\remarkname}{Remark}
\providecommand{\theoremname}{Theorem}

\newcounter{todocounter}
\newlength{\todowidthinner}
\definecolor{grey}{rgb}{0.5,0.5,0.5}

\definecolor{mhcol}{rgb}{0,0.7,0}

\definecolor{jpcol}{rgb}{0.7,0.7,0}

\begin{document}

\global\long\def\Cc{\mathbb{C}}
\global\long\def\R{\mathbb{R}}
\global\long\def\N{\mathbb{N}}
\global\long\def\Q{\mathbb{Q}}
\global\long\def\Zz{\mathbb{Z}}

\global\long\def\jump#1{\lsem#1\rsem}
\global\long\def\fc{\mathfrak{c}}
\global\long\def\fa{\mathfrak{a}}
\global\long\def\fA{\mathfrak{A}}

\global\long\def\cP{\mathcal{P}}
\global\long\def\cR{\mathcal{R}}
\global\long\def\cS{\mathcal{S}}
\global\long\def\cT{\mathcal{T}}
\global\long\def\cN{\mathcal{N}}
\global\long\def\sF{\mathscr{F}}
\global\long\def\cC{\mathcal{C}}
\global\long\def\cL{\mathcal{L}}
\global\long\def\cE{\mathcal{E}}
\global\long\def\cF{\mathcal{F}}
\global\long\def\cG{\mathcal{G}}
\global\long\def\cH{\mathcal{H}}
\global\long\def\cM{\mathcal{M}}
\global\long\def\cT{\mathcal{T}}
\global\long\def\cW{\mathcal{W}}
\global\long\def\cV{\mathcal{V}}
\global\long\def\cO{\mathcal{O}}

\global\long\def\OP{\mathcal{C}}

\global\long\def\M{M}

\newcommand{\PHk}{\Pi_{{\mathcal H}_k}}
\newcommand{\PSk}{\Pi_{{\mathcal S}_k}}

\newcommand{\ukh}{u_{{\mathcal S}_k}}
\newcommand{\const}{C_\Gamma}
\newcommand{\abs}[1]{\left\vert #1 \right\vert}
\newcommand{\Ck}{{\mathcal C}^1_{k,0}(\Omega)}
\newcommand{\CK}{{\mathcal C}^1_{K,0}(\Omega)}
\newcommand{\patch}{{G}}
\newcommand{\RPlus}{{\mathbb R}_{+}}
\newcommand{\zeroNorm}[2][]{%
  \ifstrempty{#1}{\left\Vert #2 \right\Vert_0}{\left\Vert #2 \right\Vert_{0, #1}}
}
\newcommand{\oneNorm}[2][]{%
  \ifstrempty{#1}{\left\Vert #2 \right\Vert_1}{\left\Vert #2 \right\Vert_{1, #1}}
}
\newcommand{\oneSNorm}[2][]{%
  \ifstrempty{#1}{\left\vert #2 \right\vert_1}{\left\vert #2 \right\vert_{1, #1}}
}

\newcommand{\HNorm}[2][]{%
  \ifstrempty{#1}{\left\Vert #2 \right\Vert}{\left\Vert #2 \right\Vert_{ #1}}
}

\newcommand{\aNorm}[2][]{%
  \ifstrempty{#1}{\left\Vert #2 \right\Vert_a}{\left\Vert #2 \right\Vert_{ a, #1}}
}

\newcommand{\Proj}[2][]{%
\fint_{#1}{#2}^2 \; dx
}

\title{Numerical Homogenization of Fractal Interface Problems}

\author{Ralf Kornhuber, Joscha Podlesny, and Harry Yserentant}

\address{Ralf Kornhuber, Institut f\"ur Mathematik, 
 Freie Universit\"at Berlin, 
 14195 Berlin, Germany}
\email{kornhuber@math.fu-berlin.de}

\address{Joscha Podlesny, Institut f\"ur Mathematik, 
 Freie Universit\"at Berlin, 
 14195 Berlin, Germany}
\email{podlesjo@math.fu-berlin.de}

\address{Harry Yserentant, Institut f\"ur Mathematik, 
 Technische Universit\"at Berlin, 
 10623 Berlin, Germany}
\email{yserentant@math.tu-berlin.de}

\date{}

\maketitle
\begin{abstract}
We consider the numerical homogenization of a class of fractal elliptic interface problems 
inspired by related mechanical contact problems from the geosciences.
A particular feature  is that the solution space depends on the actual fractal geometry.
Our main results concern the construction of projection operators with suitable stability and approximation properties. 
The existence of such projections then allows for the application of existing concepts from localized orthogonal decomposition (LOD)
and successive subspace correction to construct  first multiscale discretizations and iterative algebraic solvers with scale-independent convergence behavior for this class of problems.
\end{abstract}

\thanks{
This research has been funded by Deutsche Forschungsgemeinschaft (DFG)
through grant CRC 1114 ''Scaling Cascades in Complex Systems'', Project Number 235221301,
Project B01  ''Fault networks and scaling properties of deformation accumulation''.
}

\section{Introduction}
%
%
Classical homogenization aims at deriving computationally feasible, effective mathematical descriptions 
of multiscale phenomena by capturing the fine scales in terms of  local cell problems.
Starting from elliptic problems with oscillating coefficients~\cite{allaire1992homogenization,allaire1996multiscale}
and its random counterparts~\cite{JKO1994,zhikov2006homogenization}
(stochastic) homogenization  has become a flourishing field of research 
and a well-established, powerful tool in  mathematical modelling with multiple scales.
An enormous variety of applications include
multiscale materials, featuring irregular or even fractal boundaries, transmission conditions across fractal interfaces,
 or long, thin fibers~\cite{grebenkov2006mathematical,lancia2002transmission,mosco2015layered}, 
biological materials like lung tissue~\cite{baffico2008homogenization,cazeaux2015homogenization},
or polycrystals giving rise to multiscale interface problems 
with  jump conditions across a fine scale network of 
 interfaces~\cite{cioranescu2013homogenization,donato2004homogenization,gruais2017heat}.
Corresponding stochastic variants have been studied in~\cite{heida2012stochastic,Hummel1999}.

%
%
Classical  homogenization typically relies on scale separation and periodicity of fine scale behavior.
To overcome these limitations in practical computations, numerical homogenization aims at
deriving multiscale discretizations and  iterative algebraic solution methods
that are robust with respect to the inherent lack of smoothness of multiscale problems.
A natural approach to multiscale discretization is to build all relevant fine scale features of a given problem 
directly into the approximating ansatz space.
Over more than two decades, this basic idea has led to composite finite elements~\cite{hackbusch1997composite,preusser20113d}, 
variational multiscale methods~\cite{hughes1998variational}, 
heterogeneous multiscale methods~\cite{abdulle2012heterogeneous,weinan2003heterognous},
and multiscale finite elements~\cite{efendiev2009multiscale,hou1997multiscale}.
A certain breakthrough in the mathematical understanding of multiscale discretization methods
for elliptic self-adjoint problems with oscillating coefficients came with the 
seminal paper on localized orthogonal decomposition (LOD) by M{\aa}lqvist and Peterseim~\cite{maalqvist2014localization}.
Starting from a projection  $\Pi: \cH \to \cS_h$ that 
maps the solution space $\cH$ onto some given finite element space $\cS_h\subset \cH\subset L^2$
with mesh size $h$ and satisfies the following stability and approximation property
\begin{equation} \label{eq:1}
\|\Pi v \|_{\cH} \leq c \|v\|_{\cH}, \qquad \|v - \Pi v\|_{L^2}\leq C h \|v\|_{\cH} \qquad \forall v \in \cH,
\end{equation}
they observed that the a-orthogonal complement $\cW$ of the kernel of  $\Pi$ 
(the orthogonal complement with respect to the underlying energy scalar product)
has the same dimension as $\cS_h$ and,
without any additional assumptions on periodicity or scale separation,
provides an approximation with optimal accuracy.
Moreover, optimal accuracy is preserved under localization of the a-orthogonalized nodal basis of $\cW$.
The actual computation of these localized basis functions amounts to an approximate solution of local problems,
utilizing a much larger finite element space $\cS$ that resolves all fine scale features of the given problem.

An alternative to multiscale discretization methods is to use such a large  finite element space $\cS$ directly for discretization
and derive iterative algebraic solution methods that converge independently
both of  the discretization parameters and of the regularity of the continuous solution.
The construction of such methods has been carried out successfully in the framework of
iterative subspace correction~\cite{kornhuber2016numerical,xu1992iterative,xu2008uniform,yserentant1993old}.
Each iteration step typically requires the solution of a set of fully decoupled local subproblems
that capture the different frequencies of the actual error.
In particular, subspace correction methods can be applied to localization in LOD~\cite{kornhuber2018analysis}
and are often merged with multiscale discretization techniques e.g., to enhance convergence of multigrid methods
by enrichment of coarse grid spaces~\cite{hackbusch1997composite,kornhuber1994multilevel}.
%
%
While the LOD approach to the construction of multiscale discretizations makes explicit use of a projection 
$\Pi: \cH \to \cS_h$ with stability and approximation property \eqref{eq:1}, 
such kind of projections play a crucial role in the 
convergence analysis of subspace correction methods (see, e.g., \cite{kornhuber1994multilevel} and the references cited therein).
The explicit construction and analysis of such operators for standard Sobolev and finite element spaces
has  therefore quite a history with further applications in finite element convergence theory and a posteriori error analysis~\cite{brenner1994two,carstensen2006clement,clement1975approximation,ern2017finite,oswald1993bpx,verfurth1999error}.

%
%
In this paper, we consider numerical homogenization of a class of elliptic fractal interface problems
without periodicity and scale separation that is motivated by geology.
Experimental studies suggest that  grains in fractured rock are 
distributed in a fractal manner~\cite{nagahama1994scaling,turcotte1994crustal},
an observation which  is also reflected by  geophysical modelling of fragmentation 
due to tectonic deformation~\cite{sammis1986self}.
All spatial scales ranging from grains and rocks even up to tectonic plates 
are interacting in  geophysical fault networks 
that play an essential role in the dynamics of earthquake sources (see, e.g., ~\cite{rundle2003statistical} and the literature cited therein).
Mathematical modelling of stress accumulation and release in fault networks
gives rise to continuum mechanical problems with frictional contact along the interfaces 
(see, e.g. ~\cite{pipping2016efficient} and the literature cited therein).
Linearization of contact conditions leads to  elliptic interface problems,
where frictional motion along interfaces is replaced by weighted jumps of diplacement.

%
%
Scalar versions of such interface problems with fractal interface geometry 
have  recently been suggested and analyzed by Heida et al.~\cite{heida2017fractal}.
More precisely,  the fractal interface $\Gamma$ is the limit of level-$k$ interface networks $\Gamma^{(k)}$ for $k \to \infty$
and a level-$k$ interface  network $\Gamma^{(k)}= \bigcup_{j=1}^k \Gamma_j$
consists of  single faults $\Gamma_j$.
Here, the single  faults $\Gamma_j$ are ordered from "strong" to "weak"
in the sense that discontinuities of displacements along $\Gamma_j$ are expected 
to decrease for increasing $k$, because  "more fractured" media are expected to show higher 
resistance~\cite{gao2014strength,oncken2012strain}.
For each fixed $k$, the level-$k$ networks $\Gamma^{(k)}$ divide the computational domain $\Omega$ into a finite number of cells 
representing, e.g., geological grains, rocks, and plates.
For each $k \in \N$, we define a Hilbert space $\cH_k$ 
by completion of piecewise smooth functions  in $\Omega\setminus \Gamma^{(k)}$ 
with respect to a scalar product involving the broken $H^1$-seminorm 
and weighted $L^2$-norms of jumps across $\Gamma_j$, $j=1,\dots k$.
The solution space $\cH$ for interface problems  on the limiting fractal geometry $\Gamma$
is finally defined by completion of $\bigcup_{k=1}^{\infty} \cH_k$.
We consider self-adjoint elliptic  variational problems  in $\cH$.
Observe that the multiscale character of such problems goes beyond 
the usual lack of smoothness, because the solution space $\cH$ itself
depends on the actual fractal geometry which is not accessible by a fixed classical finite element space.
This suggests multiscale modifications of  classical finite elements
as ansatz spaces allowing for a priori discretization error estimates.

%
%
The main results of this paper concern the construction of projection operators $\Pi_k: \cH \to \cS_k$
with the stability and approximation property \eqref{eq:1} 
for spaces $\cS_k$ of piecewise linear finite elements  with respect to a triangulation $\cT^{(k)}$
resolving the level-$k$ interface network $\Gamma^{(k)}$, $k\in \N$.
These results allow for direct access to existing approaches to numerical homogenization, e.g., by LOD or subspace correction.
Our construction consists of two steps. We first consider projections $\Pi_{\cH_k}: \cH \to \cH_k$ and then
$\Pi_{\cS_k}: \cH_k \to \cS_k$, both with the desired properties \eqref{eq:1}. 
As projections $\Pi_{\cS_k}$ can be essentially taken from the 
literature~\cite{brenner1994two,carstensen2006clement,clement1975approximation,ern2017finite,oswald1993bpx,verfurth1999error}, 
we mainly concentrate on  the construction and analysis of $\Pi_{\cH_k}$ by extending common concepts based on local 
Poincar\'e inequalities~\cite{carstensen2006clement,verfurth1999error}.
Here, the presence of jump terms  creates various technical difficulties.
In particular, counterexamples show that it is not possible 
to bound jumps of local averages by jumps of the original functions.
Therefore, stability of $\Pi_{\cH_k}$ requires strong assumptions on the locality of $\Gamma$
that rule out, e.g. the Cantor network~\cite{heida2017fractal,turcotte1994crustal}.
The existence of suitable projections $\Pi_k$ then opens the door to a variety of existing numerical homogenization methods.
We only  consider two simple examples to fix ideas (see~\cite{podlesny20} for more advanced applications).
The application of LOD with cell-based localization by subspace correction
in the spirit of~\cite{kornhuber2018analysis,maalqvist2014localization}
provides a multiscale discretization with optimal error estimates.
Using concepts from \cite{kornhuber2016numerical},
we also present continuous and discrete versions of a two-level multigrid method with cell-based block Gauss-Seidel smoother 
and convergence rates that are independent of mesh and scale parameters.
In the concluding numerical experiments with a highly localized fractal geometry, 
we found the theoretically predicted behavior of this method.
Moreover, application to a geologically inspired crystalline structure 
illustrates the potential of our approach in future applications.

%
%
The paper is organized as follows.
The first section contains the continuous problem  formulation.
After a detailed description of the geometry of the multiscale network $\Gamma^{(k)}$, $k\in \N$,
together with some assumptions capturing its shape regularity and fractal character,
we introduce a fractal interface problem and state existence and uniqueness.
In the next section,  we discuss convergence of its $k$-scale approximation
associated with the subspaces $\cH_k \subset \cH$.
Then we introduce suitable piecewise linear finite element spaces 
$\cS_k \subset \cH_k$ for the approximation of these $k$-scale problems and state some error estimates.
The ensuing Section~\ref{sec:PROJECTIONS} is the core of the paper. 
It contains the construction and analysis of projections 
$\Pi_k = \Pi_{\cS_k}\circ \Pi_{\cH_k}$ via local Poincar\'e inequalities, a trace lemma, and quasi-interpolation.
The next two sections are devoted to first applications of these projections $\Pi_k$ 
to construct and analyze a LOD-type multiscale discretization with optimal error estimates
and a mesh- and scale-independent subspace correction method.
We finally report on some numerical experiments that illustrate our theoretical findings
and open a perspective to future practical applications.


\section{Fractal interface problems}
\subsection{Interface networks} \label{subsec:INTERFACE}
Let $\Omega\subset \R^d$, $d=1,2,3$, be a bounded domain with Lipschitz boundary $\partial \Omega$ 
that contains a countable set of mutually disjoint interfaces $\Gamma_j$, $j\in \N$. 
We assume that each interface $\Gamma_j$ is piecewise affine
with finite $(d-1)$-dimensional Hausdorff measure.
We  consider the $k$-scale interface networks $\Gamma^{(k)}$ and their fractal limit $\Gamma$, given by
\[
\Gamma^{(k)} = \bigcup_{j=1}^k \Gamma_j,\quad k\in \N,\qquad\qquad \Gamma = \bigcup_{j=1}^\infty \Gamma_j,
\]
respectively. 
Since all interfaces $\Gamma_j$, $j\in \N$, have Lebesgue measure zero in $\R^d$,
their countable union $\Gamma$ has Lebesgue measure zero as well. 
However, $\Gamma$ might have fractal (Hausdorff-) dimension $d-s$ for some $s\in(0,1)$
and infinite $(d-1)$-dimensional measure.

For each fixed $k\in \N$, the set $\Omega\setminus \Gamma^{(k)}$
consists of a finitely many mutually disjoint, open, and simply connected cells 
$G \in \Omega^{(k)}$, i.e.
\[
\Omega\setminus \Gamma^{(k)} =\bigcup_{G \in \Omega^{(k)}} G.
\]
We assume that  $\partial G = \partial \overline{G}$ (no slits) and that either 
$G \cap \partial \Omega$ has positive $(d-1)$-dimensional Hausdorff measure
or $G\cap \partial \Omega = \emptyset$.
We also assume that the cells $G\in \Omega^{(k)}$ are star-shaped in the sense that for each $G\in \Omega^{(k)}$
there is a  center $p_G \in G$ of $G$  and a continuous function $\rho_{G}$ defined on the unit sphere $S^{d-1}$ in $\R^d$
with values in $\RPlus=\{x \in \R\;|\; x \geq 0\}$ such that
\begin{equation} \label{eq:STARSHAPED}
	G = \left\lbrace p_G + r s \, | \; s \in S^{d-1}, \, 0 \leq r < \rho_G(s) \right\rbrace .
\end{equation}
Denoting
\begin{equation} \label{def:gridParametersVerfuerth}
	R_G = 2 \max\limits_{s \in S^{d-1}} \rho_{G} (s), \quad r_{G} =2 \min\limits_{s \in S^{d-1}} \rho_{G} (s)
\end{equation}
we assume that the cell partitions $ \Omega^{(k)}$, $k\in \N$, are shape regular in the sense that 
\begin{equation} \label{eq:SHG}
 \frac{R_G}{r_G}\leq \gamma \qquad \forall G \in \Omega^{(k)}\quad \forall k\in \N
\end{equation}
holds with some constant $\gamma\geq 1$. 

Introducing the subset of invariant cells 
\[
\Omega_{\infty}^{(k)}=\left\{ G\in\Omega^{(k)}\,\vert \;\,G\in\cG^{(j)}\; \forall j > k \right\} 
\]
we define the maximal size 
\begin{equation} \label{eq:NIS}
 d_{k} =\max
 \left\{ R_G\,|\;G\in\Omega^{(k)}\backslash \Omega_{\infty}^{(k)}\right\}
\end{equation}
of cells $G\in\Omega^{(k)}$ to be divided on higher levels. 
Hence, $R_G\leq d_k$ for all $G \in \Omega^{(k)}\backslash \Omega_{\infty}^{(k)}$.
Observe that $d_k$ is monotonically decreasing in $k\in \N$. We assume
\begin{equation} \label{eq:DKZERO}
d_{k}\to\;0\quad
 \text{for}\quad k\to\infty .
\end{equation}
Let $|M|\in \N \cup \{+\infty\}$ stand for the number of elements of some set $M$. Denoting 
\[
(x,y)=\{x+s(y-x)\;|\; s\in (0,1)\},
\]
we also assume that for each fixed $k \in \N$  and all  $j \in \N$ with $j  >  k$,
there is a constant $C_{k,j}\geq 0$ such that
\begin{equation} \label{eq:CDEF}
  |(x,y)\cap G \cap \Gamma_j |\leq C_{k,j} \qquad \forall G \in \Omega^{(k)}
\end{equation}
holds for almost all $x,\;y\in \Omega$. We set $C_1=1$,  $C_j=C_{1,j}$, $j=2,\dots$, and 
\begin{equation} \label{eq:RK}
r_k = \sup_{j > k}\frac{C_{k,j}}{C_{1,j}}, \qquad k\in \N.
\end{equation}
We finally assume that the interface networks $\Gamma^{(k)}$ are self-similar in the sense that
\begin{equation} \label{eq:RKSUP}
r_k C_k  \leq C_0 , \qquad \forall k \in \N
\end{equation}
holds with some constant $C_0$.

As an example, we consider a highly localized interface network in $d=2$ space dimensions. 
Let  $\Omega=(0,1)^{2}$ be the unit square and $\{e_1, e_2\}$ denote the
canonical basis in $\R^{2}$. Then the interface networks $\Gamma^{(k)}$, $k\in\N$, are inductively constructed as follows.
Let 
\[
\Gamma^{(1)}=\Gamma_1= \{\tfrac{1}{4}e_1 + (0,e_2)\}\cup \{\tfrac{1}{4}e_2 + (0,e_1)\}
 \cup\{\tfrac{1}{2}e_1 +(0, \tfrac{1}{4}e_2)\} \cup \{\tfrac{1}{2}e_2 +(0, \tfrac{1}{4}e_1)\}.
\]
For given $\Gamma^{(k)}$, $k \geq 1$, we define
\[
\tilde{\Gamma}_{k+1} = \Gamma^{(k)}\cup\{e_1 + \Gamma^{(k)}\} \cup \{e_2 + \Gamma^{(k)}\}
\]
and set $\Gamma_{k+1} = \frac{1}{4}\tilde{\Gamma}_{k+1} \setminus \Gamma^{(k)}$.
See  Figure~\ref{fig:MODELNET} for an illustration.
The resulting interface network is self-similar by construction
which can be directly extended to $d=3$ space dimensions. 
We have $d_k = \sqrt{2}\;4^{-k}$, 
$C_k=  2^k + 2^{k-1} - 2$
and $C_{k,l} = C_{l-k+1}$, $k=2, \dots$. 
Thus $r_k = 2^{1-k}$ and  \eqref{eq:RKSUP} holds with $C_0= 3$.

\begin{figure}
\includegraphics[width=3cm]{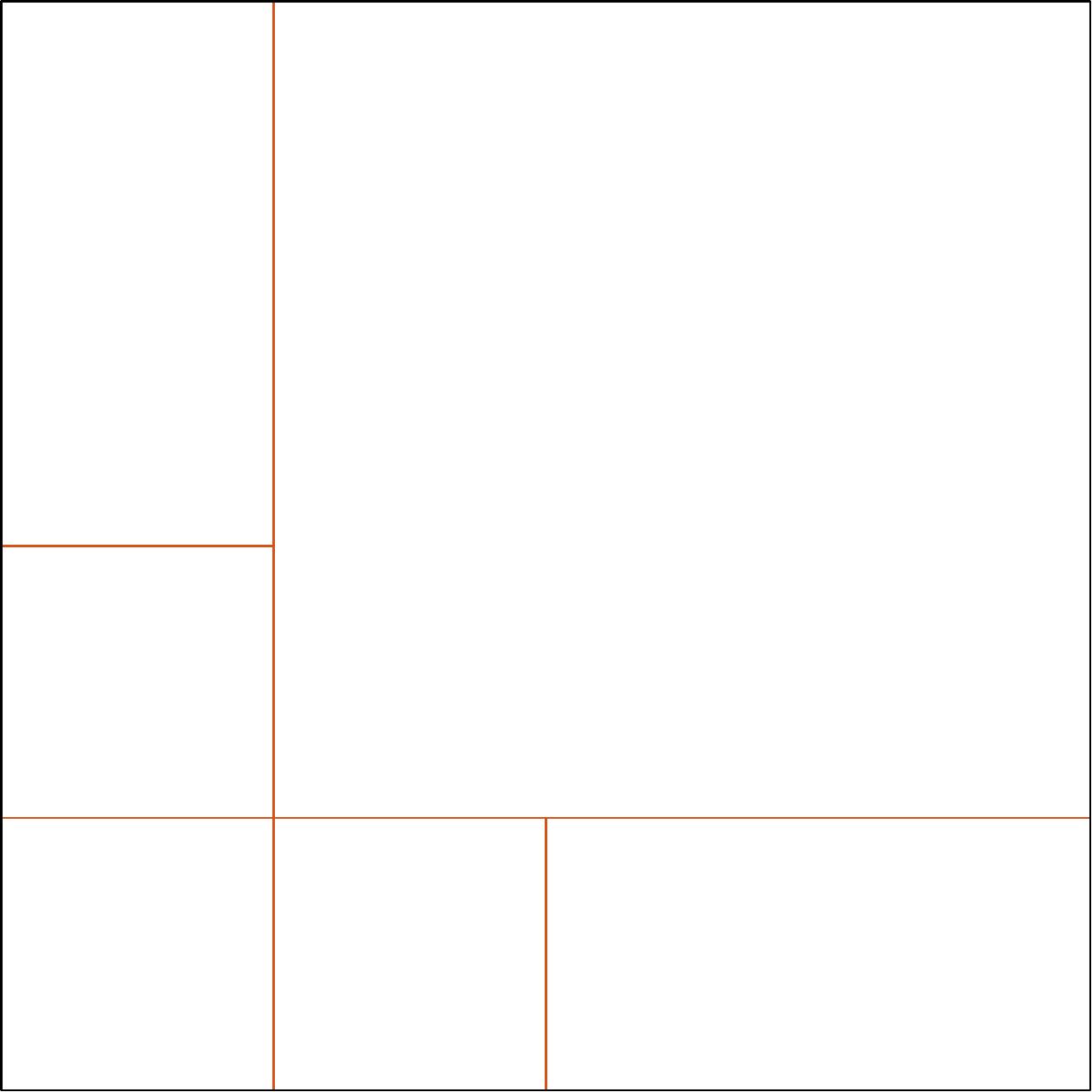}$\quad$
\includegraphics[width=3cm]{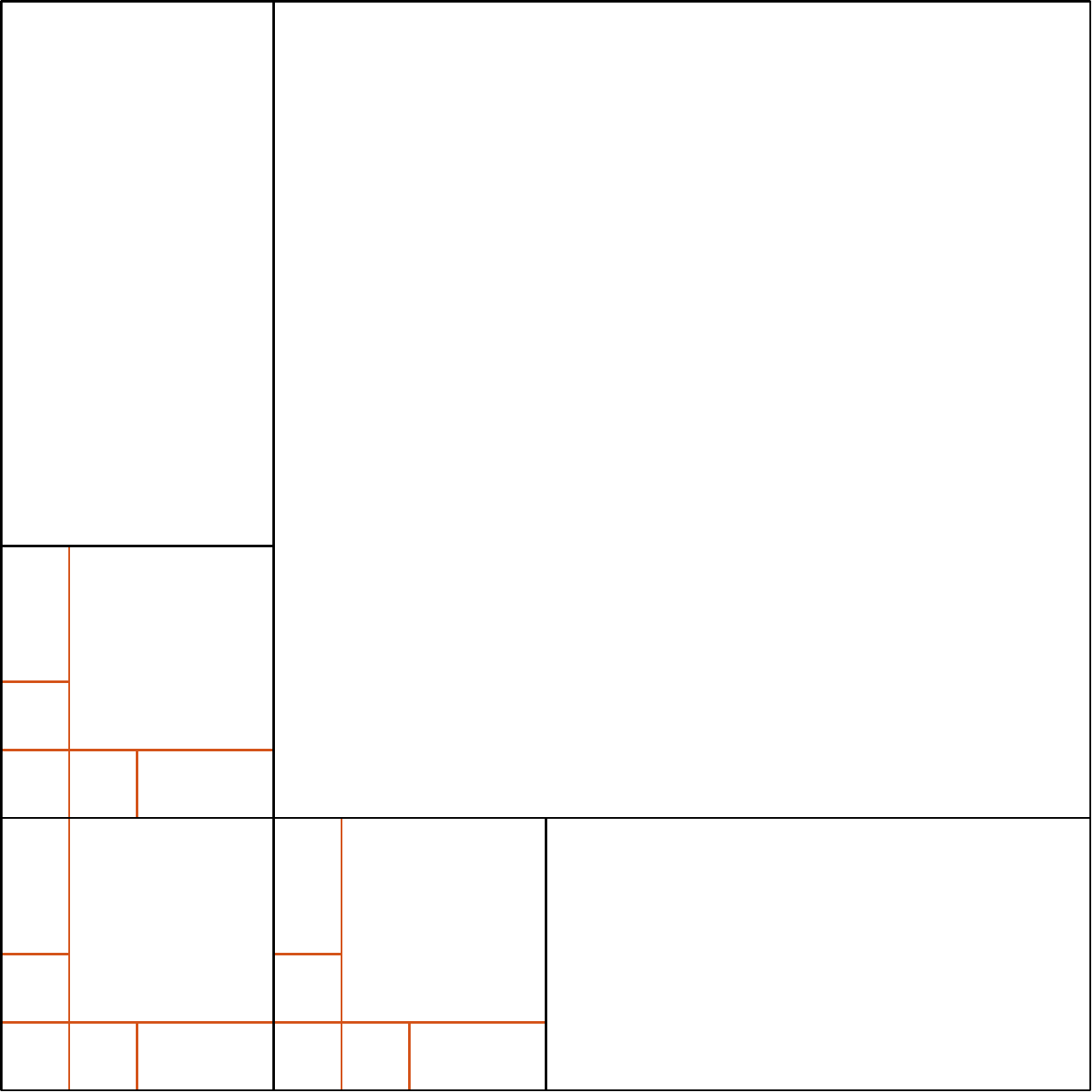}$\quad$
\includegraphics[width=3cm]{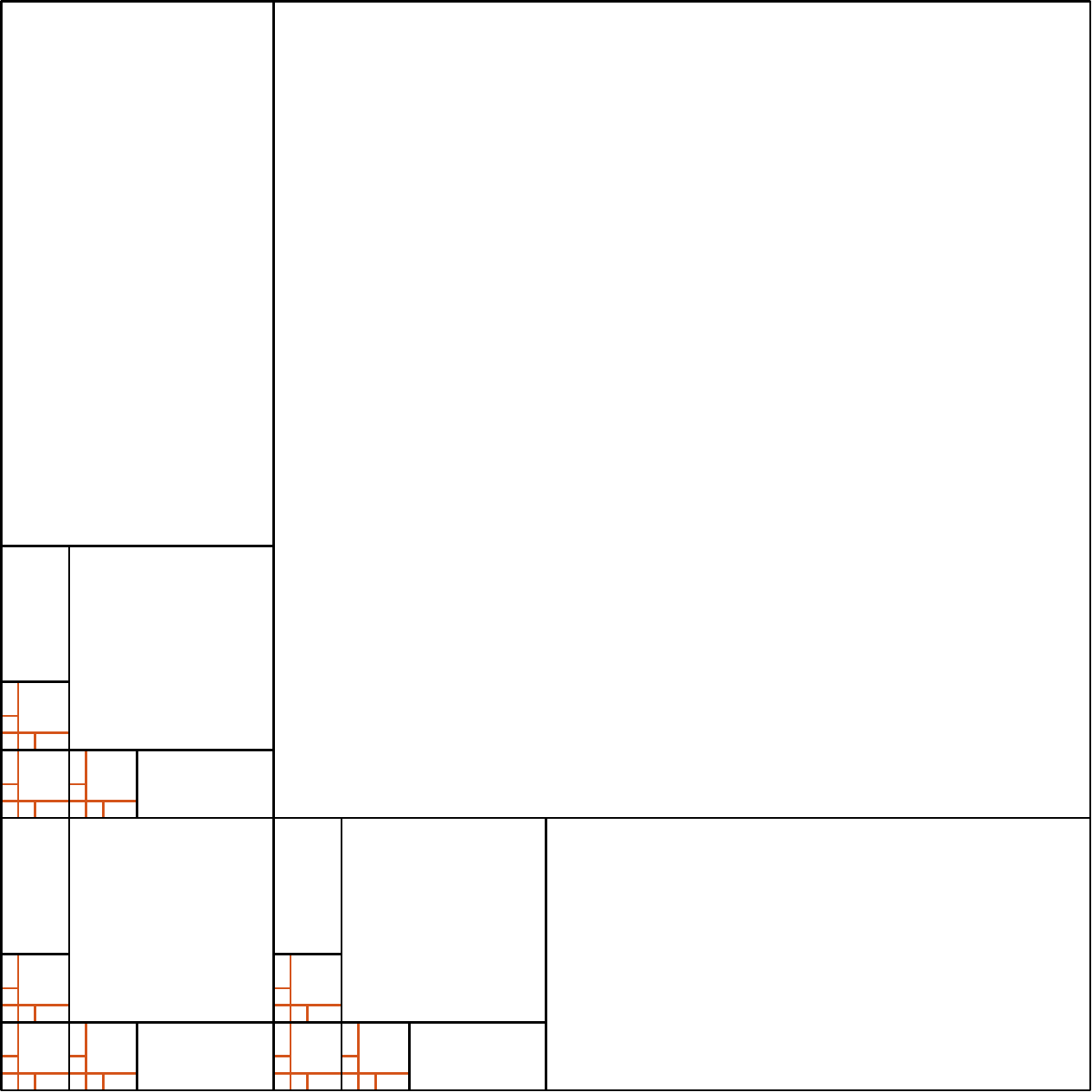}$\quad$
\includegraphics[width=3cm]{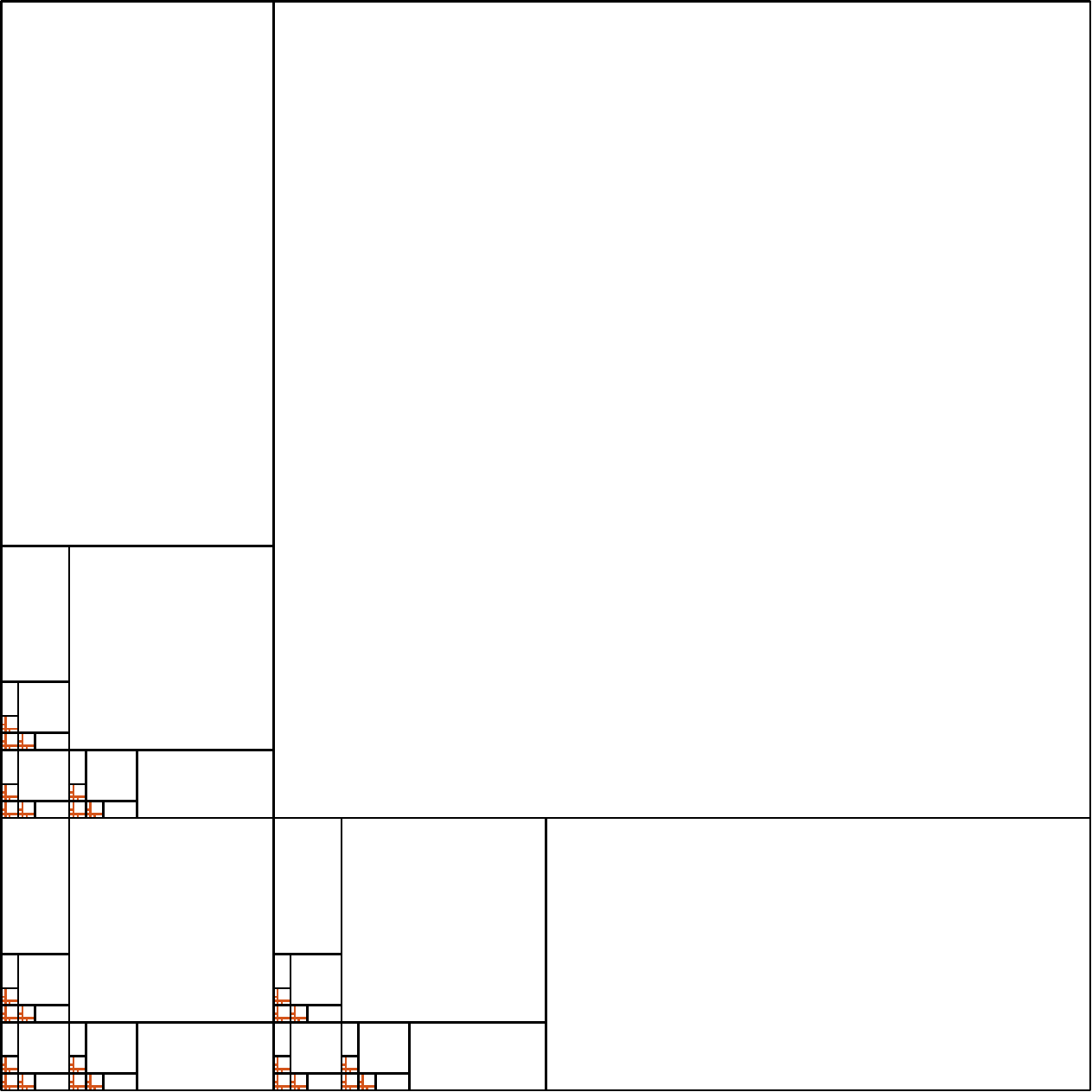}
\caption{\label{fig:MODELNET} 
Highly localized  interface network in $d=2$ space dimensions: 
 $\Gamma^{(1)}=\Gamma_1$ (red) and $\Gamma^{(k)}$ with $\Gamma_k$ (red) for $k=2,3,4$. }
\end{figure}

\subsection{Fractal function spaces}

For each fixed $k\in \N$, we introduce the space of piecewise smooth functions
\[
\Ck =\left\{
\left. v:  \overline{\Omega}\backslash \Gamma^{(k)} \to \R
\;\;  \right \vert \;\;v |_G\in C^1(\overline{G}) \; \forall G \in \Omega^{(k)}
\text{ and } v |_{\partial \Omega}\equiv0\right\}
\]
on $\Omega\backslash \Gamma^{(k)}$.
Let $j=1,\dots,k$.  
As $\Gamma_j$  is piecewise affine,
there is a normal $\nu_{\xi}$ to $\Gamma_j$ at almost all $\xi\in \Gamma_j$
and we fix the orientation of $\nu_{\xi}$ such that $\nu_{\xi}\cdot e_m>0$
with $m = \min\{i=1,\dots,d\;|\; \nu_{\xi} \cdot e_i \neq 0 \}$, and
$\{e_1,\dots,e_d\}$ denotes the canonical basis of $\R^d$.
For $\xi \in \Gamma^{(k)}$ such that $\nu_{\xi}$ exists and
for $x \neq y \in \R^d$ such that $(x-y)\cdot \nu_{\xi}\neq 0$,
the jump of $v\in C_{k,0}^{1}(\Omega)$ across $\Gamma_j$ at $\xi$
in the direction $y-x$ is defined by
\[
\jump v_{x,y}(\xi) =\lim_{s\downarrow 0}\left(v\left(\xi+s(y-x)\right)-v\left(\xi-s(y-x)\right)\right) .
\]
 Up to the sign, $\jump v_{x,y}(\xi)$
is  equal to the normal jump of $v\in C_{k,0}^{1}(\Omega)$
\[
\jump v(\xi) := \jump v_{\xi-\nu_\xi,\xi+\nu_\xi}(\xi).
\]
For some fixed material constant $\fc>0$, that, e.g., determines
the growth of resistance to jumps with increasing fracturing,
and the geometrical constant $C_j = C_{1,j}$ taken from \eqref{eq:CDEF}, we introduce the scalar product
\begin{equation} \label{eq:SCALPRO}
\left\langle v,\,w\right\rangle_k =
\int_{\Omega\backslash \Gamma^{(k)}}\nabla v\cdot\nabla w\; dx +
\sum_{j=1}^{k}\left(1+\fc\right)^{j}C_{j}\int_{\Gamma_{j}}\jump v\jump w\; d\Gamma_j ,
\quad  v, w \in C_{k,0}^{1}(\Omega),
\end{equation}
with the associated norm
$\left\Vert v\right\Vert_k = \left\langle v,\,v\right\rangle _k^{1/2}$.
Observe that $(1+ \fc)^j$ generates an exponential scaling of the
resistance to jumps across $\Gamma_j$.

Standard completion of $\Ck$ leads to a hierarchy of $k$-scale Hilbert spaces 
\[
\cH_1 \subset \cH_2 \subset \cdots \subset \cH_k, \qquad k \in \N,
\]
with the scalar products $\langle \cdot, \cdot \rangle_k$  and dense subspaces $ \Ck \subset  \cH_k$, $k \in \N$.
A limiting fractal Hilbert space $\cH$ with scalar product 
\begin{equation} \label{eq:LIMPRO}
\left\langle v,\,w\right\rangle =
\int_{\Omega\backslash \Gamma}\nabla v\cdot\nabla w\; dx +
\sum_{j=1}^{\infty}\left(1+\fc\right)^{j}C_{j}\int_{\Gamma_{j}}\jump v\jump w\; d\Gamma_j ,  \quad  v, w \in \cH,
\end{equation}
and associated norm $\HNorm{\cdot } = \langle \cdot, \cdot \rangle^{1/2}$ 
is obtained by completion of $\bigcup_{k \in \N}\cH_k$. 
We recall the main properties of $\cH$ for later use and 
refer to~\cite{heida2017fractal} for details.

The smooth subspaces $(\Ck)_{k \in \N}$, and thus the finite-scale spaces $(\cH_k)_{k\in \cH}$, 
are dense in $\cH$ in the sense that for any $v,w \in \cH$ 
there are sequences $(v_k)_{k \in \N},(w_k)_{k \in \N} \subset (\Ck)_{k \in \N}$, 
i.e., with  $v_k, w_k \in \Ck$ for all $k\in \N$, such that 
\begin{equation} \label{eq:DENSE}
\HNorm{v - v_k } \to 0,\qquad 
\langle v_k,w_k\rangle_k \to \langle v, w \rangle \quad \text{for }k\to \infty .
\end{equation}
Observe that 
\[
 \Omega\backslash \Gamma = 
 \Omega \cap (\bigcup_{j=1}^{\infty} \Gamma_j)^{\complement}\subset \Omega\backslash \Gamma^{(k)}
\]
is Lebesgue measurable so that the space $L^2(\Omega \backslash \Gamma)$ 
implicitly appearing in \eqref{eq:LIMPRO} is well-defined.
For the definition of generalized jumps $\jump v$, $v \in \cH$, 
also appearing in \eqref{eq:LIMPRO}, 
we introduce the sequence space $(L^2(\Gamma_j))_{j \in \N}$
equipped with the weighted norm
\[
\Vert z \Vert_{\Gamma}=
\left(\sum_{j=1}^{\infty} (1+ \fc)^j C_j  \zeroNorm[\Gamma_j]{z_j}^2\right)^{\frac{1}{2}}\; ,
\quad z = (z_j)_{j \in \N}\in (L^2(\Gamma_j))_{j \in \N}  ,
\]
with  $\zeroNorm[\Gamma_j]{\;\cdot\;}$ denoting the usual norm in $L^2(\Gamma_j)$.
Then, for each $v \in \cH$ and each sequence $(v_k)_{k \in \N}$ with $v_k \in \cH_k$,  the limits
 \[
\nabla v = \lim_{k\to \infty} \nabla v_k \quad \text{in }L^2(\Omega \backslash \Gamma)
\quad\text{and}\quad
\jump v = \lim_{k \to \infty}\jump {v_k} \quad \text{in }(L^2(\Gamma_k))_{k \in \N}
 \]
exist and are called weak gradient $\nabla v$ and generalized jump $ \jump{v}$ of $v$, respectively.
We have the Green's formula
\begin{equation}\label{eq:WEAK}
\int_{\Omega} v \nabla \cdot \varphi \; dx
=-\int_{\Omega\backslash\Gamma} \nabla v \cdot \varphi\; dx
+\sum_{j=1}^\infty \int_{\Gamma_j}\jump v \varphi\cdot\nu_j\; d\Gamma_j
\qquad \forall \varphi \in C_0^{\infty}(\R^d)^d 
\end{equation}
and the Poincar\'{e}-type inequality
\begin{equation}\label{eq:Poincare-H-1-2}                                      
\zeroNorm[\Omega]{v}\leq  C_P                                                                               
\left(
\oneSNorm[\Omega\backslash\Gamma]{v}^{2} +
\sum_{j=1}^{\infty}\left(1+\fc\right)^j C_{j} \zeroNorm[\Gamma_j]{\jump{v}}^2
\right)^{1/2}
\end{equation}
where  $\oneSNorm[ \Omega\setminus \Gamma ]{ v } = \zeroNorm[\Omega\setminus \Gamma]{\; \vert \nabla v \vert \;}$  and the
constant $C$ is  bounded in terms of $(1 + \frac{1}{\fc})\text{diam}(\Omega)$.
Moreover, the continuous embedding $\cH \subset H^s(\Omega)$, $s \in [0,\frac{1}{2} )$,
into Sobolev-Slobodeckij spaces $H^s(\Omega)$ (see, e.g.~\cite{slobodeckii1958generalized,triebel1982function}) allows to identify
$\cH$ with a subspace of $\bigcap_{s \in [0,\frac{1}{2} )} H^s(\Omega)$.

\subsection{Fractal interface problem}
We consider the fractal interface problem
\begin{equation} \label{eq:FRACP}
 u \in \cH: \qquad a(u,v)= (f,v) \qquad \forall v \in \cH
\end{equation}
with $f \in L^2(\Omega)$, the usual scalar product $(\cdot, \cdot)$ in $L^2(\Omega)$, and the bilinear form
\begin{equation}
a(v,w)= \int_{\Omega\backslash \Gamma}A\nabla v\cdot\nabla w\; dx +
\sum_{j=1}^{\infty}\left(1+\fc\right)^{j}C_{j}\int_{\Gamma_{j}}B\jump v\jump w\; d\Gamma_j, 
\quad v, w \in \cH,
\end{equation}
involving the functions $A: \Omega\setminus \Gamma  \mapsto  \R^{d\times d}$ 
and $B: \Gamma= \bigcup_{j=1}^{\infty}\Gamma_j \mapsto \R$.
We assume that $A(x) \in  \R^{d\times d}$ is symmetric for all $x \in \Omega\setminus \Gamma$ 
and has the properties
\begin{equation} \label{eq:PDA}
\alpha_0 |\xi|^2 \leq A(x) \xi \cdot \xi,\qquad |A(x)\xi \cdot \eta| \leq \alpha_1 |\xi||\eta|
,\qquad \forall \xi, \eta \in \R^d \qquad \forall x \in \Omega\setminus \Gamma
\end{equation}
with positive constants $\alpha_0, \alpha_1 \in \R$.
We also assume that   $B$ satisfies
\begin{equation}\label{eq:PDB}
0 < \beta_0 \leq B(x) \leq \beta_1\qquad \forall x \in \Gamma
\end{equation}
with constants $\beta_0, \beta_1\in \R$.
The assumptions \eqref{eq:PDA} and \eqref{eq:PDB} imply that $a(\cdot,\cdot)$
is symmetric and elliptic in the sense that
\begin{equation} \label{eq:ELLIP}
 \fa \HNorm{v}^2 \leq a(v,v),\qquad |a(v,w)|\leq \fA \HNorm{v}\HNorm{w}\qquad \quad \forall v,w \in \cH
\end{equation}
holds with $\fa = \min\{\alpha_0,\beta_0\}$ and $\fA = \min\{\alpha_1,\beta_1\}$.
Hence, $a(\cdot,\cdot)$ is a scalar product in $\cH$ and the associated energy norm
$\aNorm{ \cdot}= a(\cdot, \cdot)^{1/2}$ is equivalent to $\HNorm{\cdot}$.

Note that we have $(f, \cdot) \in \cH^{-1}$ due to the continuous embedding
\eqref{eq:Poincare-H-1-2} of $\cH$ into $L^2(\Omega)$.
Hence, well-posedness follows directly from the Lax-Milgram lemma.
\begin{proposition} \label{prop:FRACEXIST}
The fractal interface problem \eqref{eq:FRACP} admits a unique solution $u \in \cH$
satisfying the stability estimate
\begin{equation}
\HNorm{u} \leq  {\textstyle \frac{1}{\fa}} \, C_P \zeroNorm[\Omega]{f}. 
\end{equation}
\end{proposition}
We now focus on  the numerical approximation of the solution $u$ of the fractal interface problem \eqref{eq:FRACP}.

\section{Finite-scale discretization}
\subsection{Finite scales}
As $\cH$ is characterized by limiting properties of the $k$-scale spaces $\cH_k$, $k\in \N$,
it is natural to consider the interface problems
\begin{equation} \label{eq:MSACP}
 u_{\cH_k} \in \cH_k: \qquad a(u_{\cH_k},v)= (f,v) \qquad \forall v \in \cH_k
\end{equation}
on finite scales $k \in \N$. Note that 
\begin{equation} \label{eq:FSA}
 a(v,w)=a_k(v,w)=\int_{\Omega\backslash \Gamma}A\nabla v\cdot\nabla w\; dx +
\sum_{j=1}^{k}\left(1+\fc\right)^{j}C_{j}\int_{\Gamma_{j}}B \jump v\jump w\; d\Gamma_j, 
\quad v, w \in \cH_k.
\end{equation}
While the Lax-Milgram lemma implies existence and uniqueness,
a straightforward error estimate follows from C\'ea's lemma.
\begin{proposition} \label{prop:MULTIEXIST}
 For each $k\in \N$ the $k$-scale interface problem \eqref{eq:MSACP} 
 admits a unique solution $u_{\cH_k} \in \cH_k$
satisfying the error estimate
\begin{equation} \label{eq:CEAERR}
\HNorm{u- u_{\cH_k}} \leq  {\textstyle \frac{\fA}{\fa}} \inf_{v\in \cH_k}\HNorm{u - v} .                         
\end{equation}
\end{proposition}

In the light of \eqref{eq:DENSE} this directly implies convergence 
\begin{equation} \label{eq:FSCONV}
 \HNorm{u- u_{\cH_k} } \to 0 \qquad \text{for } k\to \infty.
\end{equation}
 In the case $A(x)= I$ and  (quite restrictive) shape regularity conditions on
$G\in \Omega^{(k)}$, $k\in \N$, there are even exponential error
estimates of the form 
\begin{equation} \label{eq:EXPERR}
\HNorm{u- u_{\cH_k} } \leq C \zeroNorm[\Omega]{f} {\textstyle\frac{1}{\fc}}(1+\fc)^{-(k-1)}
\end{equation}
with $C$ depending only on the space dimension $d$, the Poincar\'e-type constant in \eqref{eq:Poincare-H-1-2}, and shape regularity~\cite[Theorem 4.2]{heida2017fractal}.

\subsection{Finite elements on finite scales} \label{subsec:FISC}

Let $\cT^{(0)}$ be a partition of $\Omega$ 
into simplices with maximal diameter $h_{0} > 0$
which is regular in the sense that the intersection of two different simplices 
$T,\;T'\in\cT^{(0)}$ is either a common $n$-simplex for some $n=0,\dots,d-1$ or empty.
The shape regularity $\sigma>0$, i.e., the maximal ratio of  the radii of the circumscribed and the inscribed ball of $T \in \cT^{(0)}$
is preserved under uniform regular refinement~\cite{bank1983some,bey2000simplicial}.
We assume that the sequence of partitions resulting from  successive uniform regular refinement of $\cT^{(0)}$ 
resolves the interface network
in the sense that for each fixed $k\in \N$ there is a partition $\cT^{(k)}$, as obtained by a finite number of refinement steps, 
such that the interfaces $\Gamma_j$, $j=1,\dots, k$, can be represented by faces of simplices   $T\in\cT^{(k)}$, i.e.
\begin{equation} \label{eq:GAMMARES}
\Gamma^{(k)} = \bigcup_{E \in  \cE^{(k)} _{\Gamma}\subset  \cE^{(k)}} E
\end{equation}
holds with a suitable subset $\cE^{(k)}_{\Gamma}$ of  the set  $\cE^{(k)}$ of faces of simplices   $T\in\cT^{(k)}$.
In particular, this implies that 
for all $G\in \Omega^{(k)}$ the set $\cT_G^{(k)}=\{T\in \cT^{(k)}\;|\; T\subset \overline{G}\}$ is a local partition of $G$
and that the maximal diameter $h_k$ of $T\in \cT^{(k)}$  is bounded by the maximal diameter $d_k$ of $G\in \Omega^{(k)}$.
We additionally assume that $\Omega^{(k)}$ is not over-resolved in the sense that 
$d_k$ can be uniformly bounded by $h_k$, i.e., that
\begin{equation} \label{eq:HKDK}
\delta d_k \leq  h_k \leq d_k, \quad k\in \N,
 \end{equation}
holds with a constant $\delta>0$ independent of $k\in \N$.
Let $\cN_G^{(k)}$ denote the set of vertices of $T\in \cT_G^{(k)}$ that are not located on the boundary $\partial \Omega$.
Observe that each vertex located on an interface $\Gamma_j$ with two (or more) adjacent cells $G,\; G' \in \Omega^{(k)}$, 
gives rise to two (or more) different nodes $p\in \cN_G^{(k)}$ and $p'\in \cN_{G'}^{(k)}$.
For each $G \in \Omega^{(k)}$, we introduce the local finite element space $\cS_k(G)$ 
of piecewise affine functions with respect to  $\cT_G^{(k)}$
that are vanishing on $\partial G \cap \partial \Omega$.
The space $\cS_k(G)$ is spanned by the standard nodal basis $\lambda_p^{(k)}$, $p\in \cN_G^{(k)}$.
Extending these functions by zero from $\overline{G}$ to $\Omega$,
we define the broken finite element space
\[
\cS_k =\text{\rm span}\left\{ \lambda_p^{(k)}\;\vert \; p\in \cN^{(k)}\right\},\qquad  \cN^{(k)}=\bigcup_{G\in \Omega^{(k)}}\cN_G^{(k)}.
\]
The discretization of the $k$-scale interface problem~\eqref{eq:MSACP} with respect to $\cS_k$ is given by 
\begin{equation} \label{eq:FED}
\ukh\in \cS_k:\qquad a_k(\ukh,v) = (f,v)\qquad \forall v\in \cS_k\; .
\end{equation}
with $a_k(\cdot,\cdot)$ taken from \eqref{eq:FSA}.
Existence and uniqueness of the resulting finite element approximation $\ukh$ 
of $u_{\cH_k}\in \cH_k$ follows from the Lax-Milgram lemma. Convergence is implied by C\'ea's lemma together 
with \eqref{eq:FSCONV}.

\begin{proposition}
The finite  element approximations $(\ukh)_{k\in \N}$ 
converge to the solution $u$ of \eqref{eq:FRACP}
in the sense that for each $\varepsilon > 0$ 
there is a sufficiently large $k \in \N$ such that 
\begin{equation} \label{eq:ERRK}
\HNorm{u_{\cH_k} - \ukh } < \varepsilon .
\end{equation}
\end{proposition}

For each fixed $k\in \N$, the expected order of convergence 
is obtained under suitable regularity conditions on $u_{\cH_k}$.
\begin{proposition} \label{prop:ERRK}
Let $k\in \N$ and assume that $u_{\cH_k}|_G \in H^r(G)$ $\forall G \in \Omega^{(k)}$ 
with $r=2$, if $d=1,2$, and $r=2+\varepsilon$, $\varepsilon > 0$, if $d=3$.
Then the a priori error estimate
\begin{equation} \label{eq:ERRK}
\HNorm{ u_{\cH_k} - \ukh } \leq C h_k \sum_{G\in \Omega^{(k)}}\|u_{\cH_k}\|_{H^r(G)}
\end{equation}
holds with a constant $C$ depending only on the shape regularity $\sigma$ of $\cT^{(k)}$.
\end{proposition}
\begin{proof}
The proof follows from well-known interpolation error estimates~\cite{dupont1980polynomial}.
\end{proof}

A priori error estimates for the discretization error $\HNorm{u - \ukh}$ can be obtained by combining
\eqref{eq:ERRK} with exponential convergence of $u_{\cH_k}$ (see \cite[Theorem 4.2]{heida2017fractal}).
In section~\ref{sec:MULTSCALE} below, we will discuss multiscale modifications of  classical finite elements
that provide optimal a priori error estimates directly.

\section{Projections} \label{sec:PROJECTIONS}
This section is devoted to the construction of stable, surjective projections 
\[
\Pi_k:\quad \cH \to \cS_k,  \qquad k  \in \N,
\]
satisfying an approximation property. 
To this end, we extend well-known arguments~\cite{carstensen2006clement,clement1975approximation,verfurth1999error} 
to the present situation. 
\subsection{Local Poincar\'e-type inequalities}
This subsection is devoted to local Poincar\'{e}-type inequalities  on (subsets of) 
the cells $G\subset \Omega^{(k)} \setminus \Omega^{(k)}_{\infty}$
which, in contrast to cells from $\Omega^{(k)}_{\infty}$, have non-empty intersection with $\Gamma_j$ for $j>k$.
We will frequently use the notation
\[
B(G,R) = \left\lbrace p_G + r s \, | \; s \in S^{d-1}, 0 \leq r \leq R \right\rbrace
\]
for $G\in \Omega^{(k)}$ and some $R>0$.

Differences can be expressed in terms of derivatives and intermediate jumps.
\begin{lemma}\label{lem:xyDiff} \ \\ 
Let $k \in \N$,  $G\in \Omega^{(k)}\setminus \Omega^{(k)}_\infty$, 
$x, y \in G$ with  $(x,y) \subset G$ and
$\abs{(x,y) \cap \Gamma^{(k)}} < \infty$, and $K>k$. Then we have
\begin{align*}\label{eq:fundamental-estimate}
	\abs{ v(x)-v(y) }^2 
	\leq & \left( 1 + \tfrac{1}{\fc} \right)\abs{x-y}^2 \left ( \int_0^1 \abs{\nabla v \left( x+ t(y-x) \right)} \;dt \right )^2\\
	&  + \left(1 + \tfrac{1}{\fc} \right) \sum_{j=k+1}^{K} \left( 1+\fc \right)^{j-k} C_{k,j} \sum_{\xi\in (x,y) \cap \Gamma_j} \jump{v}^2 (\xi) 
			\qquad \forall v \in  \CK,
\end{align*}
where $\nabla v \left( x+ t(y-x) \right)$ is understood to be zero, if  $x+ t(y-x)\in \Gamma^{(K)}$.
\end{lemma}
\begin{proof} The assertion follows in the same way as~\cite[Lemma 3.5]{heida2017fractal}.
\end{proof}

The next lemma provides control of intermediate jumps in terms of integrals  along interfaces.
\begin{lemma} \label{lem:BALLBOUND}
Let $k \in \N$,  $B=B(G,r_G)\subset G \in \Omega^{(k)} \setminus \Omega^{(k)}_{\infty}$, and $K\geq j >k$. Then 
\begin{equation} \label{eq:BALLBOUND}
\int_{B} \int_{B}\sum_{\xi \in (x,y) \cap \Gamma_j}\jump{v}^2(\xi)\; dx \; dy \leq 
C |B| \; r_G  \int_{\Gamma_j \cap B} \jump{v}^2 \; d \Gamma_j \qquad \forall v \in \CK
\end{equation}
holds with a constant $C$ only depending on the space dimension $d$.
\end{lemma}
\begin{proof}
By similar arguments as in the proof of \cite[Theorem 3.6]{heida2017fractal},
the transformation of variables $(x,y)=\Psi(x,\eta)=(x,x+\eta)$ leads to
\[
\begin{array}{rcl}
\displaystyle \int_{B} \int_{B}\sum_{\xi \in (x,y) \cap \Gamma_j}\jump{v}^2(\xi)\; dx \; dy 
&= &\displaystyle \int_{ \{|\eta| \leq 2 r_G\} } \int_{M(\eta)}\sum_{\xi \in (x,x+\eta)}\jump{v}^2(\xi)\; dx \; d \eta\\
\leq \displaystyle  \int_{ \{|\eta| \leq 2 r_G\} } |\eta|  \int_{\Gamma_j \cap B } \jump{v}^2 \; d \Gamma_j \;d \eta
&\leq  & \displaystyle C |B|\; r_G  \int_{\Gamma_j \cap B } \jump{v}^2 \; d \Gamma_j
\end{array}
\]
with $M(\eta)=\{x \in B\;|\; x+ \eta \in B\}$ and a constant $C$ only depending on the space dimension~$d$.
\end{proof}

We are now ready to prove a Poincar\'e inequality on balls $B= B(G,r_G) \subset G \in \Omega^{(k)}\setminus \Omega^{(k)}_{\infty}$.
We will use the notation
\[
\fint_{\M}v\; dx = \frac{1}{| \M |}\int_{\M} v\; dx
\]
with suitable subsets $M \subset G$.
\begin{proposition}\label{lem:poincareBalls} 
Let $k \in \N$ and  $B=B (G,r_G) \subset G \in \Omega^{(k)} \setminus \Omega^{(k)}_{\infty}$. Then 
\begin{equation} \label{eq:PBALLS}
	\zeroNorm[ B ]{v  -   \fint_{B} v \; dx}^2  
	\leq \left( 1+ \tfrac{1}{\fc} \right) C \, r_G \left( r_G \oneSNorm[B  \setminus \Gamma]{v}^2 
	+ \sum_{j=k+1}^{\infty} \left( 1+\fc \right)^{j-k} C_{k, j} \zeroNorm[ \Gamma_j \cap  B (G,r_G)]{\jump{v}}^2  \right)
\end{equation}
holds for all $v\in \cH$  with a constant $C$ depending only on the space dimension $d$.
\end{proposition}
\begin{proof}
As $\{v \in \CK\;|\; K \in \N\}$ is dense in $\cH$ and the quantities in \eqref{eq:PBALLS} are depending continuously on $v$,
it is sufficient to prove the assertion for $v\in \CK$.  
Let $v\in \CK$ with arbitrary $K >k$ and note that the triangle inequality and Fubini's theorem imply
\begin{equation} \label{eq:DOUBLI}
	\zeroNorm[B]{v - \Proj[B]{v}}^2  =  \int_{B} \abs{ \fint_B v(x) - v(y) \, dy}^{2} \, dx .
	\leq \fint_{B} \int_{B} \abs{ v(x) - v(y)}^{2}  \, dx \, dy .
\end{equation}
Lemma \eqref{lem:xyDiff} and the Cauchy-Schwarz inequality provide
\begin{equation} \label{eq:xyDiff}
	\begin{array}{rl}
	\abs{v(x)-v(y)}^2 
	\leq & \displaystyle  \left(1+\tfrac{1}{\fc}\right) \abs{ x-y }^2 \int_0^{1} \abs{ \nabla v \left( x+ t (y-x) \right) }^2 \, dt \nonumber \\
	&   \displaystyle + \left(1+\tfrac{1}{\fc}\right) 
			\sum_{j=k+1}^{K} \left( 1+\fc \right)^{j - k} C_{k,j} \sum_{\xi\in (x, y) \cap \Gamma_j} \jump{v}^2 (\xi).
	\end{array}
\end{equation}
Treating the gradient part in the same way as in well-known proofs of the classical Poincar\'e inequality on balls 
(cf, e.g., \cite[Lemma 4.1]{evans2015measure}), we obtain 
\begin{equation} \label{eq:CBP}
\int_B \fint_B \abs{x-y}^2 \int_0^{1} \abs{ \nabla v \left( x+ t (y-x) \right) }^2 \, dt\; dy\; dx
\leq c r_G ^2 \oneSNorm[{B \setminus \Gamma^{(K)}}]{v}^2
\end{equation}
with a positive constant $c$ depending only on the space dimension $d$.
Application of  Lemma~\ref{lem:BALLBOUND} to the  jump term provides
\begin{equation} \label{eq:IJB}
\int_B \fint_B \sum_{\xi\in (x, y) \cap \Gamma_j} \jump{v}^2 (\xi) \;dy \;dx
\leq c'  r_G \int_{\Gamma_j \cap B} \jump{v}^2\; d\Gamma_j
\end{equation}
with a constant $c'$ depending only on $d$.
Inserting \eqref{eq:CBP} and \eqref{eq:IJB}  into \eqref{eq:DOUBLI}  concludes the proof.	
\end{proof}
The lines of proof of Proposition~\ref{lem:poincareBalls} carry over to the following trace analogue on spheres.
 We refer to \cite{podlesny20} for details.
\begin{lemma}  \label{lem:poincareSpheres} 
Let $k \in \N$, $B= B(G,r_G) \subset G \in \Omega^{(k)} \setminus \Omega^{(k)}_{\infty}$, and $K>k$.
Then
\begin{equation*}
		\zeroNorm[\partial B]{v - \fint_B v \; dx}^2 
		\leq \left( 1+ \tfrac{1}{\fc} \right) C  \left( r_G\oneSNorm[{B \setminus \Gamma^{(K)}}]{v}^2 
		+ \sum_{j=k+1}^K \left( 1+\fc \right)^{j-k} C_{k, j} \zeroNorm[{\Gamma_j \cap B}]{\jump{v}}^2  \right)
		\qquad \forall v \in \CK
\end{equation*}
holds with a constant $C$ depending only on the space dimension $d$.
\end{lemma}
The following lemmata prepare the extension  of the Poincar\'e inequality from balls to cells 
$G \in \Omega^{(k)}\setminus \Omega^{(k)}_{\infty}$.
We start by controlling intermediate jumps in  $G\setminus B(G,r_G)$.
 \begin{lemma} \label{lem:jumpIntegral} 
Let $k \in \N$,  $B=B(G,r_G) \subset G \in \Omega^{(k)} \setminus \Omega^{(k)}_{\infty}$, 
$ \M= G\setminus B(G,r_G) \subset G$, and $K\geq j >k$. Then we have
	\begin{equation*}
		\int_{\M} \, \sum_{\xi \in (p_G,y) \cap \Gamma_j\cap \M} \jump{v}^{2} (\xi) \, dy
		\leq  \tfrac{\gamma^{d-1}}{d}\, R_G \int_{\Gamma_j \cap \M} \jump{v}^{2} \, d \Gamma_j \qquad \forall v \in \CK.
	\end{equation*}
\end{lemma}
\begin{proof}
Assume $p_G = 0$ without loss of generality and let  $v \in \CK$ with arbitrary $K\geq j > k$. 
As  the interfaces are piecewise affine, $\Gamma_j = \bigcup_{i \in I} \Gamma_{j,i}$
can be represented as  a countable union of its affine components $\Gamma_{j,i}$, $i\in I \subset \N$.
For almost all $y \in \M$, the set  $(0, y) \cap \Gamma_j \cap \M$ is finite and we set
\begin{equation} \label{eq:DECO1}
 \sum_{\xi \in (0,y) \cap \Gamma_j \cap \M} \jump{v}^{2} (\xi) =  \sum_{i \in I}\varphi_i (y)
\end{equation} 
denoting 
\[
\varphi_i(y) =  \jump{v}^2(\xi), \quad \text{if} \quad   (0, y) \cap \Gamma_{j,i} \cap \M = \xi  \in \R^d,
\]
and $\varphi_i(y)=0$, if there is no intersection of $(0,y)$ with $\Gamma_{j,i}$ in $\M$.
We extend $\varphi_i$  by zero to the ball  $B(G, R_G)\supset G \supset \M$.   
This leads to	
\begin{equation} \label{eq:MS}
	\int_\M \varphi_i (y) \, dy 
	= \int_{B(G,R_G)\setminus B(G,r_G)} \varphi_i (y) \, dy  
	= \int_{S^{d-1}} \int_{r_G}^{R_G} \varphi_i (\Psi(r, s)) \, r^{d-1} \, dr \, ds,	
\end{equation}
where $\Psi$ stands for the transformation from  $d$-dimensional spherical to Cartesian coordinates.
We introduce the section $S_i = \lbrace s \in S^{d-1}\;|\; (0, R_G s) \cap \Gamma_{j,i} \cap M \neq \emptyset \rbrace$ 
of  directions that contribute to the integral in \eqref{eq:MS},
and ${\partial B}_i = \lbrace R_G s \;|\;\, s \in S_i \rbrace$ is the corresponding 
subset of the boundary $\partial B(G,R_G)$ of $B(G,R_G)$.
If these sets are empty or if $\Gamma_{j,i}$ is normal to ${\partial B}_i$, i.e.,  ${\partial B}_i$ is a singleton, 
then the integral in \eqref{eq:MS} vanishes.
Otherwise, there is an explicit parametrization $\xi (s) = \Psi(g_i(s)R_G, s)$ of $\Gamma_{j,i}\cap \M$ over ${\partial B}_i$ 
with a smooth function  $g_i: \,{\partial B}_i  \to (0,1]$ and, by definition,
\[
0 \leq \varphi_i (\Psi (r, s)) \leq {\jump{v}^2 \left(\xi(s) \right)}, \qquad  s \in S_i.
\]
Therefore, integration over $r$ and substitution yields
\begin{equation} \label{eq:DECO2}
\int_{S^{d-1}} \int_{r_G}^{R_G} \varphi_i (\Psi(r, s)) \, r^{d - 1}  \, dr \, ds 
\leq  \tfrac{1}{d} \, R_G \int_{S_i} \jump{v}^2 (\xi(s)) R_G^{d-1}\, ds
\leq  \tfrac{1}{d} \, R_G \int_{{\partial B}_i} \jump{v}^2 (\xi(s)) \, ds
\end{equation}
and   $g_i (s)R_G \geq r_G$, $s\in S_i$, together with  shape regularity  $R_G \leq \gamma r_G$ implies
\begin{equation} \label{eq:DECO3}
\int_{{\partial B}_i} \jump{v}^2 (\xi(s)) \, ds 
\leq  \gamma^{d-1} \int_{{\partial B}_i} \jump{v}^2 (\xi(s))  g_i^{d-2} \sqrt{g_i^2 + |\nabla g_i|^2R_G^2} \, ds 
= \gamma^{d-1} \int_{\Gamma_{j,i}\cap \M} \jump{v}^2 \, d \Gamma_{j,i}.
\end{equation}
In light of \eqref{eq:DECO1},  \eqref{eq:MS}, \eqref{eq:DECO2}, and \eqref{eq:DECO3}, summation over $i \in I$ finally leads to
\[\begin{array}{rl}
\displaystyle \int_{\M} \, \sum_{\xi \in (0,y) \cap \Gamma_j \cap \M} \jump{v}^{2} (\xi) \, dy
&  \displaystyle =  \sum_{i \in I} \int_{\M} \varphi_i (y) \, dy \\
&  \displaystyle  \leq    \tfrac{\gamma^{d-1}}{d}  \, R_G \sum_{i \in I}\int_{\Gamma_{j,i}\cap \M} \jump{v}^2 \, d \Gamma_{j,i}
=  \tfrac{\gamma^{d-1}}{d} \,  R_G \int_{\Gamma_j \cap \M} \jump{v}^{2} \, d \Gamma_j .
\end{array}
\]
\end{proof}

The next lemma is an analogue of Lemma 4.1 in~\cite{verfurth1999error}.
\begin{lemma} \label{lem:poincarePatchSplit}
Let $k \in \N$,  $G \in \Omega^{(k)}\setminus \Omega^{(k)}_{\infty}$, and $K>k$. Then 	
\begin{align*}
	\zeroNorm[G]{v}^2
 	& \leq \zeroNorm[B(G,r_G)]{v} + C R_G \zeroNorm[\partial B(G, r_G)]{v}^2   \\
	&  \quad + \left( 1 + {\textstyle \frac{1}{\fc}}\right) C R_G  \left(  R_G \oneSNorm[G \setminus \Gamma^{(K)}]{v}^2  
        +   \sum_{j=k+1}^K \left( 1+\fc \right)^{j-k} C_{k,j}  
	\zeroNorm[{\Gamma_j \cap \left( G \setminus B(G,r_G) \right)}]  {\jump{v}}^2  \right)  
 \end{align*}
holds for all  $ v \in \CK$ with a constant $C$ depending only on the dimension $d$ 
and shape regularity  $\gamma$ of $\Omega^{(k)}$.
\end{lemma}
\begin{proof}
Utilizing
\begin{equation*}
	\zeroNorm[G]{v}^2 = \zeroNorm[B(G,r_G)]{v}^2 + \zeroNorm[G \setminus B(G,r_G)]{v}^2
\end{equation*}
we have to derive a suitable bound for $\zeroNorm[G \setminus B(G,r_G)]{v}^2$.
We set $\M = G \setminus B(G,r_G)$ for notational convenience and assume $p_G = 0$ without loss of generality. 
Transformation to spherical coordinates then yields the splitting
\begin{align*}
	 \zeroNorm[\M]{v}^2
	&= \int_{S^{d-1}} \int_{r_G}^{\rho_G(s)} r^{d-1} \abs{v(r s)}^2 \, dr \, ds \\
	&=  \int_{S^{d-1}} \int_{r_G}^{\rho_G(s)} r^{d-1} \abs{v(r s) - v(r_G s) + v(r_G s) }^2 \, dr \, ds \\
	&\leq  \underbrace{2 \int_{S^{d-1}}  \int_{r_G}^{\rho_G (s)} r^{d-1} \abs{v(r s) - v( r_G s)}^2 \, dr \, ds }_{:= I_1} 
	+ \underbrace{2 \int_{S^{d-1}}  \int_{r_G}^{\rho_G(s)} r^{d-1} \abs{v(r_G s)}^2 \, dr \, ds }_{:= I_2} .
\end{align*}
We will provide suitable bounds for these two parts and first consider $I_1$.
Lemma \ref{lem:xyDiff} leads to
\begin{equation} \label{eq:L1}
\begin{array}{rl}
	I_1
	\leq & \displaystyle 2  \left( 1 + \tfrac{1}{\fc} \right) \int_{S^{d-1} }\int_{r_G}^{\rho_G(s)} r^{d-1}
	\left( \int_{r_G}^{r}  \abs{\nabla v(z s)} \, dz \right)^2 \, dr \, ds \\
	& \displaystyle   \quad + 2  \left( 1 + \tfrac{1}{\fc} \right) \int_{S^{d-1}} \int_{r_G}^{\rho_G (s)} r^{d-1}
	 \sum_{j=k+1}^{K} \left(1+\fc \right)^{j - k} C_{k,j} \sum_{\xi \in (r_G s, r s) \cap \Gamma_j} \jump{v}^2 (\xi) \, dr \, ds.
	 \end{array}
\end{equation}
By the Cauchy-Schwarz inequality and straightforward computations, 
as in the proof of  \cite[Lemma 4.1]{verfurth1999error},
the gradient term  in \eqref{eq:L1} can be bounded according to
\begin{equation} \label{eq:L1GRAD}
\begin{array}{rl} 
&\displaystyle \int_{S^{d-1} }\int_{r_G}^{\rho_G(s)} r^{d-1} \abs{ \int_{r_G}^{r} \nabla v(z s) \, dz }^2 \, dr \, ds \\
\leq &  \displaystyle 		\int_{S^{d-1}} \left( \int_{r_G}^{\rho_G(s)} z^{d-1} \abs{\nabla v(z s)}^2 \, dz \right) 
		\left( \int_{r_G}^{\rho_G(s)} r^{d-1}  \int_{r_G}^{r} z^{1-d} \, dz \, dr  \right) \, ds \\
 \leq &c R_G^2 \oneSNorm[\M\setminus \Gamma^{(K)}]{v}^2 
\end{array}
\end{equation}
with a constant $c$ depending only on the dimension $d$ and shape regularity $\gamma \geq \frac{R_G}{r_G}$ of $\Omega^{(k)}$.
In order to bound the  jump contributions in \eqref{eq:L1} in terms of integrals along interfaces,
we apply Lemma \ref{lem:jumpIntegral} to obtain
\begin{equation} \label{eq:L1JUMP}
\begin{array}{rl}
	& \displaystyle  \int_{S^{d-1}} \int_{r_G}^{\rho_G (s)} r^{d-1} \sum_{\xi\in (r_G s, r s) \cap \Gamma_j} \jump{v}^2 (\xi) \, dr \, ds 
	    = \int_{\M} \sum_{\xi\in (0, y) \cap \Gamma_j \cap \M} \jump{v}^2 (\xi) \, dy \\
		& \displaystyle \leq   \tfrac{\gamma^{d-1}}{d} \, R_G \int_{\Gamma_j\cap \M} \jump{v}^2  \, d\Gamma_j
		=   \tfrac{\gamma^{d-1}}{d} R_G\zeroNorm[{\Gamma_j\cap \M}]{\jump{v}}^2.
\end{array}
 \end{equation}
 Inserting $M=G \setminus B(G,r_G) \subset G$,  the estimates \eqref{eq:L1GRAD} and \eqref{eq:L1JUMP} provide
\begin{equation} \label{eq:I1BOUND}
	I_1  \leq 
	2 c\left( 1 + \tfrac{1}{\fc} \right) \left(  R_G^2 \oneSNorm[G \setminus \Gamma^{(K)}]{v}^2  + 
		  \tfrac{\gamma^{d-1}}{d} R_G\sum_{j=k+1}^{K} \left( 1+\fc \right)^{j- k} C_{k,j} \zeroNorm[\Gamma_j\cap (G \setminus B(G,r_G))]{\jump{v}}^2\right) .
\end{equation}
 Straightforward calculation leads to
\begin{equation} \label{eq:I2BOUND}
	\begin{array}{rl}
		I_2
		& \displaystyle = 2 \int_{S^{d-1}} \int_{r_G}^{\rho_G(s)} r^{d-1} \abs{ v (r_G s) }^2 \, dr \, ds
	      = 2 \int_{S^{d-1}} r_G^{d-1} \abs{ v (r_G s) }^2 \int_{r_G}^{\rho_G (s)} \left( \tfrac{r}{r_G} \right)^{d-1} \, dr \, ds \\
		& \displaystyle = 2 \int_{S^{d-1}} r_G^{d-1} \abs{ v (r_G s) }^2 \tfrac{r_G}{d} \left( \left( \tfrac{\rho_G(s)}{r_G} \right)^d - 1 \right) \, ds
	     \leq \tfrac{2}{d}  \left( \left( \tfrac{R_G}{r_G} \right)^d - 1 \right)  r_G \zeroNorm[\partial B(G,r_G)]{v}^2 .
	\end{array}
\end{equation}
Together with \eqref{eq:I1BOUND} this concludes the proof.
\end{proof}

As a direct extension  of Lemma~4.3 in~\cite{verfurth1999error},
we are now ready to state a local Poincar\'e inequality on cells $G\in \Omega^{k}\setminus \Omega^{(k)}_{\infty}$.

\begin{proposition} \label{prop:poincarePatch} \ \\
For every $k\in \N$ and every cell $G\in \Omega^{(k)}\setminus \Omega^{(k)}_{\infty}$, the local Poincar\'e inequality 
\begin{equation} \label{eq:POINCG}
	\zeroNorm[G]{v -  \fint_{G}v\; dx}^2
	\leq C \left( 1 + \tfrac{1}{\fc} \right) d_k \left( d_k \oneSNorm[{G \setminus \Gamma}]{v}^2 
	+ \sum_{j=k+1}^{\infty} \left( 1+\fc \right)^{j-k} C_{k,j} \zeroNorm[{\Gamma_j \cap G}]{\jump{v}}^2 \right)
\end{equation} 
holds for all $ v\in \cH$ with a constant $c$  depending only on the dimension $d$ and shape regularity  $\gamma$ of $ \Omega^{(k)}$.
\end{proposition}
\begin{proof}
 It is sufficient to show  \eqref{eq:POINCG} for $v\in \CK$ with arbitrary $K > k$, and then use a density argument.
Observe that  $\fint_G v\;dx$ is minimizing the functional $\zeroNorm[G]{v - {}\cdot{}}^2$.
Denoting $B=B(G,r_G)$, we conclude from Lemma \ref{lem:poincarePatchSplit}
\begin{align*}
	 \zeroNorm[G]{v - \fint_G v\; dx}^2 &\leq \zeroNorm[G]{v - \fint_B v \;dx}^2 \\		  
	  & \leq \zeroNorm[B]{v - \fint_B v \;dx}^2 + C R_G \zeroNorm[\partial B]{v - \fint_B v \;dx}^2 \\
 	   &  +   C R_G  \left( 1 + \tfrac{1}{\fc} \right) \left( R_G \oneSNorm[G \setminus \Gamma^{(K)}]{v}^2
 	 +   \sum_{j=k+1}^K \left( 1+\fc \right)^{j-k} C_{k,j}  
 	    \zeroNorm[{\Gamma_j \cap \left( G \setminus B \right)}]{\jump{v}}^2 \right) .
\end{align*}
Now the assertion follows from the Poincar\'e inequality on balls stated in  Proposition~\ref{lem:poincareBalls}
together with its trace analogue for spheres Lemma~\ref{lem:poincareSpheres}.
\end{proof}

\subsection{A trace lemma} \label{subsec:TRACE}
In order to control the jump contributions in the stability estimates below,
we provide some estimates of traces on the interfaces $\Gamma_j$ of  functions $v \in \CK$ 
with arbitrary $K\in \N$.
For this purpose, we follow the approach by Verf\"urth~\cite{verfurth1999error}
and utilize the triangulations $\cT^{(k)}$ introduced in Subsection~\ref{subsec:FISC}.
The following lemma is a direct extension of \cite[Lemma 3.2]{verfurth1999error} 
and can be shown  along the same lines of proof.
The additionally arising jump contributions are controlled in a similar way as 
in Lemma~\ref{lem:BALLBOUND} and \cite[Theorem 3.6]{heida2017fractal}.
 We refer to \cite{podlesny20} for details.

\begin{lemma} \label{lem:traceLemma} 
Let $k \in \N$, $T\in \cT^{(k)}$, and $E \in \cE^{(k)}$ be a face of $T$. Then 	
\begin{align*}
	\zeroNorm[E]{v}^2
		& \leq   c  \left( 1 + \tfrac{1}{\fc}\right)  \left( h_k^{-1} \zeroNorm[T]{v}^2  + h_{k} \oneSNorm[T \setminus \Gamma^{(K)} ]{v}^2 
		 +  \sum_{j=k+1}^{K} \left( 1+\fc \right)^{j-k} C_{k,j} \zeroNorm[{\Gamma_j \cap T}]{ \jump{v}}^2 \right)
	\end{align*}
holds for all $v \in \CK$ with $K>k$ and a constant $c$ depending only on the space dimension $d$ and
shape regularity $\sigma$ of $\cT^{(k)}$.
\end{lemma}
Now we are ready to state the desired trace lemma.
\begin{lemma} \label{cor:traceCorollary} \ \\
Let $k \in \N$ and $G  \in \Omega^{(k)}\setminus \Omega^{(k)}_{\infty}$ and $l=1,\dots,k$. Then
\[
	\zeroNorm[\Gamma_l \cap \partial G] {v}^2
	 \leq C \left( 1 + \tfrac{1}{\fc}\right)  \left(d_k^{-1} \zeroNorm[G]{v}^2 + d_k \oneSNorm[G \setminus \Gamma^{(K)}]{v}^2 
	+ \sum_{j=k+1}^{K} \left( 1+\fc \right)^{j - k} C_{k,j} \zeroNorm[\Gamma_j \cap G]{ \jump{v}}^2 \right)
\]
holds for all $v \in \cH_K$ with $K>k$ and a constant $C$ depending only on the space dimension~$d$,
shape regularity $\sigma$ of $\cT^{(k)}$ and the constant $\delta$ in \eqref{eq:HKDK}.
\end{lemma}
\begin{proof}
By a density argument, it is sufficient to consider $v \in \CK$.
Let $G \in  \Omega^{(k)}\setminus \Omega^{(k)}_{\infty}$ and recall that $\cT^{(k)}_G \subset \cT^{(k)}$ 
is a local partition of $\cE_G^{(k)}\subset \cE^{(k)}$.
Denoting the set  faces of simplices $T\in \cT^{(k)}_G$ by $\cE^{(k)}_G$, 
select the subset  of faces $\cE^{(k)}_{\partial G} \subset \cE^{(k)}_G$
such that
\[
\partial G = \bigcup_{E \in \cE^{(k)}_{\partial G}} E.
\]
Note that for each $E \in \cE^{(k)}_{\partial G}$ there is a simplex $T_E\in \cT^{(k)}_G$ with face $E$ and
a simplex $T\in \cT^{(k)}_G$  can contribute at most all of its $d+1$ faces to $\cE^{(k)}_{\partial G}$.
Utilizing the  trace Lemma~\ref{lem:traceLemma}  and \eqref{eq:HKDK}, we get
\begin{align*}
		& \zeroNorm[\Gamma_l \cap \partial G]{v}^2
		\leq \sum_{E \in \cE^{(k)}_{\partial G}} \zeroNorm[E]{v}^2 \\ 
		& \leq c \left( 1 + \tfrac{1}{\fc}\right)  \sum_{E \in \cE^{(k)}_{\partial G}}
		\left(h_k^{-1}  \zeroNorm[T_E]{v}^2 + h_k \oneSNorm[T_E  \setminus \Gamma^{(K)}]{v}^2 +  \sum_{j=k+1}^K \left( 1+\fc \right)^{j - k} C_{k,j} \zeroNorm[\Gamma_j \cap T_E]{ \jump{v}}^2 \right) \\
		& \leq c (d+1)   \left( 1 + \tfrac{1}{\fc}\right)  \sum_{T \in \cT^{(k)}_G}
		\left( h_k^{-1} \zeroNorm[T]{v}^2 + h_k \oneSNorm[T \setminus \Gamma^{(K)}]{v}^2 + \sum_{j=k+1}^K \left( 1+\fc \right)^{j - k} C_{k,j} \zeroNorm[\Gamma_j \cap T]{ \jump{v}}^2 \right) \\
		& \leq C  \left( 1 + \tfrac{1}{\fc}\right) 
		\left( d_k^{-1} \zeroNorm[G]{v}^2 + d_k \oneSNorm[G \setminus \Gamma^{(K)}]{v}^2 +  \sum_{j=k+1}^K \left( 1+\fc \right)^{j - k} C_{k,j} \zeroNorm[\Gamma_j \cap G]{ \jump{v}}^2 \right) 
\end{align*}
with a constant $c$ depending only on the space dimension $d$,
shape regularity $\sigma$ of $\cT^{(k)}$, and the constant $\delta$  in \eqref{eq:HKDK}.
\end{proof}

\subsection{Projections on finite-scale spaces $\cH_k$}
\begin{definition}
For every $k\in \N$, we define the linear projection $\PHk: \cH \to \cH_k$ by setting
\begin{equation} \label{eq:PIHKDEF}
 \PHk v|_G = \left\{
    \begin{array}{ll}
        \underset{v_k\in H^1(G)}{\text{ \rm arg min}}\{|v-v_k|_{1,G\setminus \Gamma}\;|\; \int_G v-v_k\; dx = 0\}, & G \in \Omega^{(k)}\setminus \Omega_{\infty}^{(k)}\\
        v|_G, & G \in \Omega_{\infty}^{(k)}
    \end{array}
    \right .
    \qquad 
\end{equation}
for all $G \in \Omega^{(k)}$ and $v\in \cH$.
\end{definition}
The operator  $\PHk$ is well-defined. Indeed, for every $G\in \Omega^{(k)}\setminus \Omega_{\infty}^{(k)}$ 
its local contribution $v_k$ is the unique solution of a quadratic minimization problem on the affine space 
$\fint_Gv\; dx + W$, $W=\{w \in H^1(G)\;|\; \int_Gw\; dx = 0\}$, which is characterized by the variational equality
\begin{equation} \label{eq:VAREQ}
 (\nabla v_k,\nabla w)=(\nabla v, \nabla w) \qquad \forall w \in W.
\end{equation}

\begin{lemma}\label{lem:LOCSTAB}
 For every $k\in \N$ the linear projection $\PHk$ satisfies
 \begin{equation} \label{eq:MEANLOCSTAB}
 \fint_G v - \PHk v \; dx = 0 \quad \text{and}\quad 
  \oneSNorm[G]{\PHk v}\leq \oneSNorm[G\setminus \Gamma]{v}\qquad \forall v\in \cH.
 \end{equation}
\end{lemma}
\begin{proof}
 Setting $v_k=\PHk v|_G$, the first equality follows by definition \eqref{eq:PIHKDEF} 
 and after testing with $w= v_k - \fint_Gv\; dx$ in \eqref{eq:VAREQ},
 the remaining local stability of $\PHk$ follows from the Cauchy-Schwarz inequality.
\end{proof}

We now state an approximation property of the projections $\PHk v$, $k \in \N$.
\begin{theorem} \label{thm:HKAPPROX}
Assume that the condition
\begin{equation} \label{eq:SMALLCL}
 r_k(1+\fc)^{-k}\leq d_k
\end{equation}
on the geometry of the interface network $\Gamma$ is satisfied.
Then the projections $\PHk: \cH \to \cH_k$, $k \in \N$,  have the approximation property
\begin{equation}
\zeroNorm{v - \PHk v}^2  \leq  c \left( 1 + \tfrac{1}{\fc} \right) d^2_k  \HNorm{v}^2\qquad \forall v\in \cH
\end{equation}
with a constant $c$ depending only on the space dimension $d$ and shape regularity $\gamma$ of $ \Omega^{(k)}$.
\end{theorem}
\begin{proof}
Let $G \in  \Omega^{(k)}\setminus \Omega_{\infty}^{(k)}$ and $v\in \cH$.
As $v - \PHk v$  has mean-value zero and $\PHk v$ does not jump across $\Gamma_l$ for $l \geq k+1$,
the local Poincar\'e inequality stated in  Proposition~\ref{prop:poincarePatch}  yields 
\begin{equation} \label{eq:POINAPPL}
\zeroNorm[G]{v - \PHk v}^2 \leq 
c\left(1+ \tfrac{1}{\fc}\right) d_k
\left( 
d_k \oneSNorm[G \setminus \Gamma]{v - \PHk v}^2  + \sum_{j=k+1}^{\infty} (1+ \fc)^{j-k} C_{k,j}\zeroNorm[\Gamma_j\cap G]{\jump{v}}^2
\right)
\end{equation}
with a constant $c$  depending only on the dimension $d$ and shape regularity  $\gamma$ of $\Omega^{(k)}$.
Assumption \eqref{eq:SMALLCL} and the definition \eqref{eq:RK} of $r_k$ imply
\begin{equation} \label{eq:GEOGA}
(1+\fc)^{-k}C_{k,j}\leq r_k(1+\fc)^{-k}C_j\leq d_k C_j.
\end{equation}
Now we insert these estimates into \eqref{eq:POINAPPL} 
and make use of the Cauchy-Schwarz inequality and of  the local stability \eqref{eq:MEANLOCSTAB} to obtain
\[
\zeroNorm[G]{v - \PHk v}^2 \leq 
c\left(1+ \tfrac{1}{\fc}\right) d_k^2
\left( 
 2\oneSNorm[G \setminus \Gamma]{v}^2  + \sum_{j=k+1}^{\infty} (1+ \fc)^j C_j\zeroNorm[\Gamma_j \cap G]{\jump{v}}^2
\right).
\]
As $\zeroNorm[G]{v - \PHk v}= 0$ for all $G\in \Omega^{(k)}_{\infty}$, 
summation over $G\in \Omega^{(k)}\setminus  \Omega_{\infty}^{(k)}$ completes the proof.
\end{proof}

For each fixed $k \in \N$ boundedness
\begin{equation} \label{eq:CONTP}
\HNorm{\PHk v } \leq \mu_k \HNorm{v} \qquad \forall v \in \cH
\end{equation}
of $\PHk$ holds with a constant $\mu_k$ 
as a consequence of the closed graph theorem~\cite[????]{podlesny20}.
In order to identify sufficient conditions for uniform stability of $\PHk$,
we want to further clarify the dependence  of $\mu_k$ on $k\in \N$.
To this end, the following lemma provides a bound for the jump contributions to $\HNorm{\PHk v}$ in terms of $\HNorm{v}$. 
\begin{lemma} \label{lem:cstabHelp} 
Let $k \in \N$,  $G \in \Omega^{(k)}\setminus \Omega^{(k)}_{\infty}$ and
assume that conditions \eqref{eq:HKDK} and \eqref{eq:SMALLCL} are satisfied. 
Then
	\begin{equation*}
		\sum\limits_{l=1}^{k} (1+\fc)^l C_l \zeroNorm[{\Gamma_l}]{\jump{v - \PHk v}}^2
		\leq C\left( 1 + \tfrac{1}{\fc}\right)^2  d_k \left(\sum\limits_{l=1}^{k} (1+\fc)^l C_l \right)  \HNorm{v}^2.
\end{equation*}
holds for all $v\in \CK$ with $K>k$ and
a constant $C$ depending only on the space dimension~$d$, shape regularity $\sigma$ of $\cT^{(k)}$, 
and the constant $\delta$ in \eqref{eq:HKDK}.
\end{lemma}
\begin{proof}
Let $k \in \N$  and $v\in \CK$ with $K>k$. Note that 
\[
\begin{array}{rcl}
	 \displaystyle \zeroNorm[\Gamma_l]{\jump{v - \PHk v}}^2
		& = &  \displaystyle\sum_{\substack{G, G'  \in \Omega^{(k)} \\ G \neq G'}} \int_{\Gamma_l \cap \partial G \cap \partial G'} 
		\left( (v - \PHk v)\vert_G - (v- \PHk v)\vert_{G'} \right)^2 \, d\Gamma_l \\
		& \leq & \displaystyle 4 \sum \limits_{G \in \Omega^{(k)} } \zeroNorm[\Gamma_l \cap \partial G]{v - \PHk v}^2 .
\end{array}
\]
holds for $ l =1,\dots, k$.
Inserting \eqref{eq:GEOGA} (a consequence of assumption \eqref{eq:SMALLCL})
into the local  approximation property~\eqref{eq:POINAPPL},  we get 
\begin{equation} \label{eq:LOCAPPASS}
\zeroNorm[G]{v - \PHk v}^2 \leq 
c\left(1+ \tfrac{1}{\fc}\right) d_k^2
\left( 
\oneSNorm[G \setminus \Gamma^{(K)}]{v - \PHk v}^2  
+ \sum_{j=k+1}^{K} (1+ \fc)^j C_j \zeroNorm[\Gamma_j \cap G]{\jump{v}}^2 \right) .
\end{equation}
As $\jump{\PHk v} = 0 $ on $\Gamma_j$ for $j > k$, application of the trace Lemma~\ref{cor:traceCorollary}, 
together with  \eqref{eq:LOCAPPASS}, Lemma~\ref{lem:LOCSTAB},  and \eqref{eq:GEOGA} lead to
\begin{align*}
	& \zeroNorm[\Gamma_l \cap \partial G]{v- \PHk v}^2 \\
	\leq &  c' \left( 1 + \tfrac{1}{\fc}\right)  \left( d_k^{-1} \zeroNorm[G]{v-\PHk v}^2
	+ d_{k} \oneSNorm[G \setminus \Gamma^{(K)}]{v - \PHk v}^2 +  \sum_{j=k+1}^{K} \left( 1+\fc \right)^{s-k} C_{k,j} \zeroNorm[\Gamma_j\cap G]{ \jump{v}}^2 \right) \\
	\leq & C'\left( 1 + \tfrac{1}{\fc}\right)^2  d_k \left( 
	 \oneSNorm[G \setminus \Gamma^{(K)}]{v}^2 +   \sum_{j=k+1}^{K} \left( 1+\fc \right)^{j} C_j \zeroNorm[\Gamma_j\cap G]{ \jump{v}}^2 \right)
\end{align*}
with constants $c', C'$ depending on 
the space dimension~$d$,
shape regularity $\gamma$ of $\Omega^{(k)}$,
shape regularity $\sigma$ of $\cT^{(k)}$,
and the constant $\delta$ in \eqref{eq:HKDK}.
Summation over $G \in \Omega^{(k)}$  yields
\[
\zeroNorm[\Gamma_l ]{v- \PHk v}^2  \leq  C \left( 1 + \tfrac{1}{\fc}\right)^2  d_k \HNorm{v}^2
\]
and the assertion follows.
\end{proof}

We are ready to state stability of the projections $\PHk$, $k \in \N$.
\begin{theorem}\label{thm:HKSTAB}
Assume that  conditions  \eqref{eq:HKDK} and  \eqref{eq:SMALLCL} are satisfied.
Then the projections  $\PHk: \cH \to \cH_k$, $k \in \N$, are stable in the sense that
\begin{equation} \label{eq:PHKSTAB}
\HNorm{\PHk v}^2 \leq  c\left( 1 + \tfrac{1}{\fc}\right)^3 d_k \left(\sum\limits_{l=1}^{k} (1+\fc)^l C_l \right)  \HNorm{v}^2
\qquad \forall v\in \cH
\end{equation}
holds for each $k\in \N$ with a constant $c$ depending only on the space dimension $d$,  
shape regularity $\gamma$ of $\Omega^{(k)}$, shape regularity $\sigma$ of $\cT^{(k)}$, 
and the constant $\delta$ in \eqref{eq:HKDK}.
\end{theorem}
\begin{proof}
As $\CK$, $K\in \N$, is dense in $\cH$ and $\PHk$ is continuous for each fixed $k\in \N$,
it is sufficient to prove  \eqref{eq:PHKSTAB}  for $v \in \CK$ with arbitrary $K\geq k$. 
In light of
\[
\HNorm{\PHk v} \leq \| v - \PHk v\|_{\cH_k} + \HNorm{v}
\]
it is sufficient to derive a corresponding bound for $\HNorm{ v - \PHk v }$.
Utilizing boundedness of $\PHk$ with respect to  $| {}\cdot{} |_{1, \Omega \setminus \Gamma}$, cf.  Lemma~\ref{lem:LOCSTAB},
and that, by construction, $\PHk v$ is does not jump across $\Gamma_l$, $l > k$,
we obtain	
\[
\begin{array}{rcl}
	\HNorm{v - \PHk v}^2 &=& \displaystyle  | v - \PHk v |_{1, \Omega\setminus \Gamma}^2 +
	\left( 1 + \tfrac{1}{\fc} \right) \sum_{l=1}^{K} \left( 1+\fc  \right)^{l} C_{l} \zeroNorm[\Gamma_l]{ \jump{v - \PHk v}}^2\\
		&\leq& \displaystyle 4 \HNorm{v}^2 
		+ \left( 1 + \tfrac{1}{\fc} \right) \sum_{l=1}^{k} \left( 1+\fc \right)^{l} C_{l} \zeroNorm[\Gamma_l]{ \jump{v - \PHk v}}^2.
\end{array}
\]
Now the assertion follows from Lemma \ref{lem:cstabHelp}.
\end{proof}

Uniform stability of $\PHk$ is obtained under an additional condition on the geometry of the interface network $\Gamma$.
\begin{corollary} \label{cor:UNISTAB}
Assume that conditions \eqref{eq:HKDK} and \eqref{eq:SMALLCL} are satisfied
and that the additional condition
\begin{equation} \label{eq:QDEC}
 d_k\left( \sum_{l=1}^k(1+ \fc)^l C_l \right) \leq \const, \qquad k \in \N,
\end{equation}
holds with a constant $\const$ independent of $k$.
Then the projections $\PHk$, $k \in \N$, are uniformly stable, i.e., 
\begin{equation} \label{eq:PHKUNISTAB}
\HNorm{\PHk v} \leq  c  \HNorm{v} \qquad \forall v\in \cH
\end{equation}
holds for each $k\in \N$ with a constant $c$ depending only on the space dimension $d$,  
shape regularity $\gamma$ of $\Omega^{(k)}$, shape regularity $\sigma$ of $\cT^{(k)}$, 
the constant $\delta$ in \eqref{eq:HKDK},
the constant $\const$ in \eqref{eq:QDEC}, and the material constant $\fc$.
\end{corollary}

The additional condition \eqref{eq:QDEC} reflects the fact that the jump contributions to $\HNorm{\PHk v}$
cannot be bounded by the jump contributions to $\HNorm{v}$ (see \cite[???]{podlesny20} for a simple counterexample).
Relating the material constant $\fc$ to the geometry of the interface network,
it  implies that the  interfaces $\Gamma^{(k)}$ are highly localized for feasible $\fc>0$
and thus excludes, e.g., the Cantor network~\cite{heida2017fractal,turcotte1994crustal,turcotte1997fractals}.
For example, the highly localized network described in Subsection~\ref{subsec:INTERFACE} above
satisfies condition \eqref{eq:QDEC} for  $\fc \leq 1$.

\subsection{Quasi-interpolation on finite element spaces $\cS_k$}
We now  construct and analyse suitable  projections $\PSk: \cH_k\to \cS_k$, utilizing well-known concepts 
from finite element analyis.
\begin{definition}
For every $k\in \N$, we define the  Cl\'ement-type quasi-interpolation\\ ${\PSk:\cH_k \to \cS_k}$ by setting
\begin{equation} \label{eq:PSKDEF}
 \PSk v = \sum_{p\in \cN^{(k)}} \left(\Pi_pv\right)\; \lambda_p^{(k)}
\end{equation}
with $ \Pi_p: \cH_k \to \R$ defined by
\begin{equation}
\Pi_p v =  \fint_{\omega_p} v\; dx,\qquad {\omega_p} = \text{\rm supp }\lambda_p^{(k)},\quad p\in \cN^{(k)},
\end{equation}
for $v\in \cH_k$.
\end{definition}

\begin{proposition} \label{prop:SKAPPROX}
Let $k\in \N$ and $G \in \Omega^{(k)}$. Then the  projection $\PSk$ defined in \eqref{eq:PSKDEF} 
has the local approximation property
\begin{equation} \label{eq:DAPPROX}
 \zeroNorm[G]{v - \PSk v}  \leq  c  h_k  \oneSNorm[G]{v} \qquad \forall v\in \cH_k
\end{equation}
with a constant $c$ depending only on the dimension $d$ and shape regularity $\sigma$ of $\cT^{(k)}$.
\end{proposition}
\begin{proof}
 Let $v\in \cH_k$, $G\in \Omega^{(k)}$, and $T\in \cT_G^{(k)}\subset \cT^{(k)}$. Then 
 \[
 \zeroNorm[T]{v - \PSk v}^2\leq C h_k^2\sum_{p\in T\cap \cN_G^{(k)}} \oneSNorm[\omega_p]{v}^2
 \]
holds  with a constant $C$ depending only on the dimension $d$ and shape regularity $\sigma$ of 
$\cT^{(k)}$~\cite{carstensen2006clement,verfurth1999error}.
The assertion now follows by summation over $T \in \cT_G^{(k)}$.
\end{proof}

\begin{proposition} \label{prop:SKSTAB}
The  projections $\PSk$, $k\in \N$, defined in \eqref{eq:PSKDEF} 
are stable in the sense that
\begin{equation} \label{eq:DAPPROX}
 \HNorm{\PSk v}  \leq  c   d_k \left(\sum_{l=1}^k (1+ \fc)^l C_l\right) \HNorm{v} \qquad \forall v\in \cH_k
\end{equation}
holds with a constant $c$ depending only on the dimension $d$ and shape regularity $\sigma$ of $\cT^{(k)}$.
\end{proposition}
\begin{proof}
 Let $v \in \cH_k$ and observe that 
 \begin{equation} \label{eq:TRIANGC}
 \HNorm{\PSk v}^2 \leq 2 \HNorm{v}^2 +  \oneSNorm[\Omega\setminus \Gamma^{(k)}]{\PSk v}^2
 + 2 \sum_{l=1}^k (1+ \fc)^l C_l \zeroNorm[\Gamma_l]{\jump{v- \PSk v}}^2 
 \end{equation}
 follows from the triangle inequality and the Cauchy-Schwarz inequality.
It is well-known, e.g., from~\cite[Theorem 2.4]{carstensen2006clement} that 
 \begin{equation} \label{eq:HSKLOCB}
 \oneSNorm[\Omega\setminus \Gamma^{(k)}]{\PSk v}^2 = \sum_{G\in \Omega^{(k)}} \oneSNorm[G]{\PSk v}^2
 \leq c \sum_{G\in \Omega^{(k)}} \oneSNorm[G]{v}^2 = c \oneSNorm[\Omega\setminus \Gamma^{(k)}]{v}^2 \leq c\HNorm{v}^2 
 \end{equation}
 holds with a constant $c$ depending only on shape regularity $\sigma$ of $\cT^{(k)}$ and the space dimension~$d$.
 We now derive a corresponding bound for the jump terms occurring in \eqref{eq:TRIANGC}.
As $\cT^{(k)}$ resolves the interface network  $\Gamma^{(k)}$ according to \eqref{eq:GAMMARES}, 
there are subsets $\cE^{(k)}_l\subset \cE^{(k)}$ such that
\[
\Gamma_l= \bigcup_{E \in \cE^{(k)}_l} E,\quad l= 1,\dots,k.
\]
Now let  $E \subset \overline{G}_{E,1}\cap \overline{G}_{E,2} \subset \Gamma_l$ with $G_{E,i}\in \Omega^{(k)}$, $i=1,2$,
and we set $v_i = v|_{G_{E,i}}$, $i=1,2$. Then the Cauchy-Schwarz inequality yields
\begin{equation} \label{eq:JUMPSEP}
\zeroNorm[\Gamma_l]{\jump{v- \PSk v}}^2 = \sum_{E \in \cE^{(k)}_l}\int_{E}\jump{v- \PSk v}^2\; d E
\leq 2 \sum_{E \in \cE^{(k)}_l}\int_{E}|v_1- \PSk v_1|^2 + |v_2- \PSk v_2|^2\; dE.
\end{equation}
It is well-known~\cite{carstensen2006clement,verfurth1999error} that
\[
\int_{E}|v_i- \PSk v_i|^2 \; dE \leq c \sum_{p\in \cN_{E,i}} h_k \oneSNorm[\omega_p]{v_i}^2, \quad i=1,2,
\]
holds with $\omega_p=\text{\rm supp }\lambda_p$, 
$\cN_{E,i} = E \cap \cN^{(k)}_{G_{E,i}}$ denoting the vertices of $E$ located in $\overline{G}_{E,i}$,
and a constant $c$ depending only on shape regularity $\sigma$ of $\cT^{(k)}$ and the space dimension 
$d$.
After inserting this bound into \eqref{eq:JUMPSEP}, summation over $l=1,\dots,k$, and  shape regularity of $\cT^{(k)}$
leads to
\[
\sum_{l=1}^k (1+ \fc)^l C_l \zeroNorm[\Gamma_l]{\jump{v- \PSk v}}^2 \leq c h_k \left(\sum_{l=1}^k (1+ \fc)^l C_l\right) \oneSNorm{v}^2
\]
with $c$ only depending on $\sigma$ and $d$ and the assertion follows from \eqref{eq:HKDK}.
\end{proof}

Note  that  uniform stability of $\PSk$, $k\in \N$, is obtained under the additional assumption \eqref{eq:QDEC}. 

\begin{definition}
 For every $k\in \N$, we define the projection 
 \begin{equation}\label{eq:PIDEF}
  \Pi_k = \PSk \circ \PHk: \cH \to \cS_k.
 \end{equation}
\end{definition}

\begin{theorem} \label{theo:UNIAPPROXALL}
Assume  that the conditions \eqref{eq:HKDK},  \eqref{eq:SMALLCL}, \eqref{eq:QDEC} hold.
Then the projections  $\Pi_k: \cH \to \cS_k$, $k\in \N$, defined in \eqref{eq:PIDEF} have the approximation property
\begin{equation} \label{eq:APPROX}
 \zeroNorm{v - \Pi_k v}  \leq  c  h_k  \HNorm{v} \qquad \forall v\in \cH
\end{equation}
with a constant $c$ depending only on the space dimension $d$, 
shape regularity $\gamma$ of $\Omega^{(k)}$, shape regularity  $\sigma$ of  $\cT^{(k)}$,
the constant $\delta$ in \eqref{eq:HKDK},  the constant $\const$ in \eqref{eq:QDEC},
and the material constant~$\fc$.
\end{theorem}
\begin{proof}
The assertion is an immediate consequence of the triangle inequality
\[
\zeroNorm{v - \Pi_k v}\leq \zeroNorm{v - \Pi_{\cH_k} v} + \zeroNorm{\Pi_{\cH_k} v - \Pi_{\cS_k} \left(\Pi_{\cH_k}  v \right)},
\]
Theorem~\ref{thm:HKAPPROX}, Proposition~\ref{prop:SKAPPROX}, and Corollary~\ref{cor:UNISTAB}.
\end{proof}

Uniform stability of the projections $\Pi_k$ is an immediate consequence 
of Corollary~\ref{cor:UNISTAB} and Proposition~\ref{prop:SKSTAB}.
\begin{theorem} \label{theo:UNISTABALL}
Assume  that the conditions \eqref{eq:HKDK},  \eqref{eq:SMALLCL}, \eqref{eq:QDEC} hold.
Then the  projections $\Pi_k: \cH \to \cS_k$, $k\in \N$, defined in \eqref{eq:PIDEF} are uniformly stable in the sense that
\begin{equation} \label{eq:APPROX}
 \HNorm{\Pi_k v}  \leq  c  \HNorm{v} \qquad \forall v\in \cH
\end{equation}
holds with a constant $c$ depending only on the space dimension $d$, 
shape regularity $\gamma$ of $\Omega^{(k)}$, shape regularity  $\sigma$ of $\cT^{(k)}$, 
the constant $\delta$ in \eqref{eq:HKDK}, the constant $\const$ in \eqref{eq:QDEC},
and the material constant~$\fc$.
\end{theorem}

\section{Multiscale finite element discretization} \label{sec:MULTSCALE}
For some fixed $k \in \N$, we now construct novel multiscale finite element spaces
with the same dimension as $\cS_k$ that provide discretization errors of order $h_k$.
Utilizing the projection $\Pi_k: \cH \to \cS_k$ defined in \eqref{eq:PIDEF},
we can readily apply local orthogonal decomposition (LOD) 
as introduced by M{\aa}lqvist \& Peterseim~\cite{maalqvist2014localization}
with  localization by subspace decomposition as suggested in~\cite{kornhuber2018analysis}.

Let $\cV_k = \text{ker } \Pi_k \subset \cH$  denote the kernel of $\Pi_k$ and $\OP: \cH \to \cV_k$ the orthogonal projection
of $\cH$ onto $\cV_k$ with respect to the scalar product $a(\cdot, \cdot)$ in $\cH$. Then the multiscale finite element space
\[
\cW_k = \{ v - \OP v\;|\; v \in \cH \}= \{ v - \OP v\;|\; v \in \cS_k \} = \text{span}\{(I-\OP)\lambda_p^{(k)}\;|\; p\in \cN_k\}
\]
is isomorphic to $\cS_k$. We consider the multiscale discretization
\begin{equation} \label{eq:WANALGLOB}
u_k \in \cW_k: \qquad a(u_k,v) = (f, v)\qquad \forall v \in \cW_k .
\end{equation}
The following error analysis is due to Peterseim~\cite{peterseim2016variational} 
and~M{\aa}lqvist \& Peterseim~\cite{maalqvist2014localization} (see also~\cite{kornhuber2018analysis}).
\begin{theorem} \label{theo:CDISCRET}
The unique solution $u_k$ of the discrete problem~\eqref{eq:WANALGLOB} 
is given by
\begin{equation} \label{eq:WSOLREP}
u_k= (I- \OP) \Pi_k u.
\end{equation}
The discretization error has the representation $u-u_k = Cu$ and the error estimate
\[
\HNorm{u-u_k} \leq c h_k \zeroNorm{f} 
\]
holds with $c$ depending only on the constants appearing in Theorems~\ref{theo:UNIAPPROXALL}, \ref{theo:UNISTABALL},
 and the ellipticity constant $\fa$ from \eqref{eq:ELLIP}.
\end{theorem}
In spite of these desired properties, the space $\cW_k$ is problematic, 
because its multiscale basis functions $(I-\OP)\lambda_p^{(k)}$, $p\in \cN_k$, in general have global support.
We therefore consider (intensionally local) approximations $\OP_{\nu}: \cH \to \cH$, $\nu \in \N$,  of $\OP$
giving rise to the approximate subspaces
\[
\cW_k^{(\nu)} = \text{span}\{(I-\OP_\nu)\lambda_p^{(k)}\;|\; p\in \cN\}
\]
and corresponding Galerkin discretizations
\begin{equation} \label{eq:WANAL}
u_k^{(\nu)} \in \cW_k^{(\nu)}: \qquad a(u_k^{(\nu)},v) = (f, v)\qquad \forall v \in \cW_k^{(\nu)} .
\end{equation}
The following discretization error estimate is taken from~\cite{kornhuber2018analysis}.

\begin{theorem} \label{theo:CAPPROX}
Assume that the approximations  $\OP_{\nu}: \cH \to \cH$, $\nu \in \N$,  of $\OP$ 
are convergent in the sense that 
\begin{equation} \label{eq:CNVCL}
\aNorm{\OP v - \OP_{\nu} v } \leq q \aNorm{\OP v }, \qquad  \nu \in \N_k,
\end{equation}
holds  for all $v \in \cH$ with some convergence rate $q <1$.
Then we have the discretization error estimate
\begin{equation} \label{eq:DISCAPPROXW}
\HNorm{ u - u_k^{{(\nu)}} } \leq  \left(1 + q^{\nu}\right) \tfrac{\fA}{\fa}\HNorm{u - u_k} + q^{\nu}\tfrac{\fA}{\fa}\HNorm{u - \Pi_k u}, \qquad \nu \in \N.
\end{equation}
\end{theorem}
\begin{proof}
Exploiting $(I - \OP_\nu ) \Pi_k u \in \cW_k^{(\nu)}$ and \eqref{eq:WSOLREP}, we obtain
\[
\aNorm{u - u_k^{(\nu)} } \leq \aNorm{ u - \left (I - \OP_\nu \right) \Pi_k u } = \aNorm{(u-u_k) - ( \OP \Pi_k u - \OP_\nu \Pi_k u)}.
\]
Convergence \eqref{eq:CNVCL} together with  identity \eqref{eq:WSOLREP} provides
\[
\aNorm{\OP \Pi_k u - \OP_\nu \Pi_k u } \leq q^{\nu} \aNorm{ \OP \Pi_k u }\leq q^{\nu}(\aNorm{u - u_k }+ \aNorm{ u - \Pi_k u } ).
\]
Now the assertion follows from the triangle inequality and the norm equivalence~\eqref{eq:ELLIP}.
\end{proof}

We now concentrate on the construction of  convergent local approximations  
$\OP_{\nu}: \cH \to \cH$, $\nu \in \N$,  by local subspace correction.
Here, we make heavy use of the fact that the kernel $\cV_k$ of $\Pi_k$ is  high-frequency.
Locality   \eqref{eq:PIHKDEF}, \eqref{eq:PSKDEF} of the 
projection $\Pi_k = \PSk \circ \PHk$ motivates  the splitting 
\begin{equation} \label{eq:SPLITTING}
\cV_k = \sum_{G \in \Omega^{(k)}}  \cV_G
\end{equation}
into the subspaces
\[
\cV_G = \{ (I-\Pi_k) v|_G \;|\; v \in \cH \} \subset \cV_k, \qquad G \in \Omega^{(k)}.
\] 
Here,  $v|_G$ is defined by $v|_G(x) = v(x)$ for $x\in G$ and $v|_G(x)=0$ otherwise.
Note that the linear mapping $\cH \ni v \to v|_G\in \cH$ 
is uniformly bounded  in $\cH$  for  all $G \in \Omega^{(k)}$ and each fixed $k\in \N$
as a consequence of the trace Lemma~\ref{cor:traceCorollary}
and the continuous embedding of $\cH$ into $L^2(\Omega)$. 
The subspaces $\cV_G$ are closed, 
because convergence of a sequence $(v_i)_{i \in \N}\subset \cV_G \subset \cV_k$
to some $v \in \cH$ implies $v \in \cV_k$, i.e., $\Pi_k v= 0$, as $\cV_k$ is closed,
$v = v|_G$, as $\text{supp }v_i\subset G$ for all $i \in \N$,
and therefore $v = (I - \Pi_k)v|_G \in \cV_G$.
The following lemma is the main result of this section.
\begin{lemma} \label{lem:SPLITTSTAB}
The splitting \eqref{eq:SPLITTING} is stable in the sense that for each $v\in \cV_k$ there  is a decomposition 
$(v_G)_{G \in \Omega^{(k)}}$ of $v$ with $v_G \in \cV_G$, $G \in \Omega^{(k)}$, such that
\begin{equation} \label{eq:K1}
\sum_{G \in \Omega^{(k)}} \aNorm{v_G}^2 \leq K_1 \aNorm{v}^2
\end{equation}
holds with a constant $K_1$ depending only on
the constants appearing  in Theorems~\ref{theo:UNIAPPROXALL}, \ref{theo:UNISTABALL},
the geometric constant $C_0$ in \eqref{eq:RKSUP} and  the ellipticity constants $\fa$, $\fA$ from \eqref{eq:ELLIP}.

Assume that for all $k \in \N$ and each $G$ in $\Omega^{(k)}$ the number of neighboring cells of $G$ from $\Omega^{(k)}$ is uniformly bounded by $c_N \in \R$.
Then the  splitting \eqref{eq:SPLITTING} is  bounded in the sense that for each $v\in \cV_k$  
all decompositions  $(v_G)_{G \in \Omega^{(k)}}$  of $v$ with $v_G \in \cV_G$, $G \in \Omega^{(k)}$, satisfy
\begin{equation} \label{eq:K2}
\aNorm{v}^2 \leq K_2 \sum_{G \in \Omega^{(k)}} \aNorm{v_G}^2 
\end{equation}
with a constant $K_2$ depending only on $c_N$.
\end{lemma}
\begin{proof}
Boundedness \eqref{eq:K2} 
with a constant $K_2$ depending only on the maximal number of neighbors of each cell $G$ 
is a direct consequence of the Cauchy-Schwarz inequality. 

By a density argument, it is sufficient to show  \eqref{eq:K1} for  $v \in \cV_k \cap \cH_K$.
We consider the splitting of $v$ into its local components
\[
v_G = (I - \Pi_k)v|_G\in \cV_G, \qquad G \in \Omega^{(k)}.
\]
Exploiting the locality of $\Pi_k$, i.e., $(I - \Pi_k)(v|_G) = \left((I - \Pi_k)v\right)|_G$, we have
\begin{equation} \label{eq:VGSU}
\HNorm{v_G}^2 = 
\oneSNorm[G \setminus \Gamma^{(K)}]{v - \Pi_k v}^2 
+ \sum_{j=1}^k (1+\fc)^j C_j \zeroNorm[\Gamma_j\cap \partial G]{v - \Pi_k v}^2
+ \sum_{j=k+1}^K (1+ \fc)^j C_j\zeroNorm[\Gamma_j \cap G]{\jump{v}}^2.
\end{equation}
As a consequence of the Cauchy-Schwarz inequality, Lemma~\ref{lem:LOCSTAB}
and  the local boundedness~\eqref{eq:HSKLOCB} of $\PSk$, we have 
\begin{equation} \label{eq:LOCBS}
\oneSNorm[G \setminus \Gamma^{(K)}]{v - \Pi_k v}^2 \leq C \oneSNorm[G \setminus \Gamma^{(K)}]{v}^2
\end{equation}
with a constant $C$ depending only on the space dimension $d$ and shape regularity $\sigma$ of $\cT^{(k)}$.
After utilizing the trace Lemma~\ref{cor:traceCorollary}, we apply local boundedness  \eqref{eq:LOCBS},
and the geometric conditions  \eqref{eq:RKSUP}, \eqref{eq:SMALLCL}, and \eqref{eq:QDEC} to obtain 
\begin{align*}
\sum_{j=1}^k(1+\fc)^j C_j \zeroNorm[\Gamma_j\cap \partial G]{v - \Pi_k v}^2
& \leq C' \left( d_k^{-2} \zeroNorm[G]{v - \Pi_k v}^2 + 
\oneSNorm[G \setminus \Gamma^{(K)}]{v}^2 + 
\sum_{j=k+1}^K (1+\fc)^j C_j \zeroNorm[\Gamma_j \cap G]{\jump{v}}^2  \right)
\end{align*}
with  $C'$ additionally depending on the material constant $\fc$, the constant $\delta$ in \eqref{eq:HKDK} and the constants
appearing in  \eqref{eq:RKSUP} and  \eqref{eq:QDEC}.
After inserting the above estimates in \eqref{eq:VGSU}, summation over $G$ and \eqref{eq:HKDK} lead to
\[
\sum_{G \in \Omega^{(k)}}\HNorm{v_G}^2 \leq  C'' \left( h_k^{-2} \zeroNorm{v - \Pi_k v}^2 + \HNorm{v}^2 \right ).
\]
Now the approximation property stated in Proposition~\ref{theo:UNIAPPROXALL} 
together with the norm equivalence~\eqref{eq:ELLIP} concludes the proof.
\end{proof}

Let $P_G : \cH \to \cV_G$, $G \in \Omega^{(k)}$, denote the a-orthogonal Ritz projections defined by
\begin{equation} \label{eq:LOCRITZ}
P_G w \in \cV_G: \qquad a(P_G w, v) = a(w, v)  \qquad \forall  v \in \cV_G
\end{equation}
for $w \in \cH$ and 
\[
T = \sum_{G \in \Omega^{(k)}}P_G
\]
the resulting preconditioner. Lemma~\ref{lem:SPLITTSTAB} implies 
\begin{equation} \label{eq:KOND}
1/K_1 a(v,v)  \leq a(Tv,v) \leq K_2 a(v,v) \qquad \forall v \in \cV_k
\end{equation}
or, equivalently, the bound $\kappa \leq K_1 K_2$ of the condition number 
$\kappa= \aNorm{T}\aNorm{T^{-1}}$  of $T$ restricted to $\cV_k$.
We consider straightforward damped Richardson iteration 
\begin{equation} \label{eq:RICHARDSON}
\OP_{\nu + 1} = \OP_\nu + \omega T( I  - \OP_\nu), \qquad \OP_0 = 0,
\end{equation}
with a  suitable damping factor $\omega$. Note that $\OP_{\nu}v \in \cV_k$, $\nu \in \N$, holds for any $v \in  \cH$.
Now  convergence  of  \eqref{eq:RICHARDSON} follows by well-known arguments.
\begin{theorem} \label{theo:RICHARDSON}
Assume that for all $k \in \N$ and each $G$ in $\Omega^{(k)}$ the number of neighboring cells of $G$ from $\Omega^{(k)}$ is uniformly bounded by $c_N \in \R$.
Then the approximations $\OP_\nu$, $\nu \in \N$, of $\OP$ defined in \eqref{eq:RICHARDSON} 
are convergent for  $\omega < 2/K_2$  in the sense of \eqref{eq:CNVCL},
and we have  $q= 1 -1/K_1K_2$ for the optimal damping factor $\omega = 1/K_2$
with $K_1$, $K_2$ depending only on
the constants appearing  in Theorems~\ref{theo:UNIAPPROXALL}, \ref{theo:UNISTABALL},
the geometric constant $C_0$ in \eqref{eq:RKSUP}, $c_N$, and  the ellipticity constants $\fa$, $\fA$ from \eqref{eq:ELLIP}.
\end{theorem}
More sophisticated iterative schemes with better convergence rates are discussed, e.g.,  in \cite{kornhuber2018analysis}.

Utilizing Theorems~\ref{theo:CDISCRET} and \ref{theo:CAPPROX}, the desired discretization error estimate
\[
\HNorm{u- u_k^{{(\nu)}} } = \cO(h_k)
\]
is obtained by choosing $\nu \in \N$ such that the stopping criterion  $q^\nu \tfrac{\fA}{\fa}\HNorm{u -  \Pi_k u} = \cO(h_k)$ is fulfilled.

Note that the support of the first iterate $(I - \OP_1) \lambda_p^{(k)} = (I  -  \omega T) \lambda_p^{(k)}$ 
is contained in $\overline{G}$, if $p$ is located in $G$ and contained in $\overline{G}\cup \overline{G'}$, 
if $p\in \overline{G}\cap \Gamma_k\cap \overline{G'}$.
Similarly,  the support of the approximate multiscale basis functions 
$(I - \OP_\nu ) \lambda_p^{(k)} = (I  -  \omega T)^\nu \lambda_p^{(k)})$, $p \in \cN_k$, 
spreads  at most by one layer of cells in each iteration step
and therefore depends logarithmically on the prescribed accuracy of order $h_k$.

The construction of $\cW_k^{(\nu)}$ requires the successive solution of local problems~\eqref{eq:LOCRITZ} 
in the infinite dimensional function spaces $\cV_G$.
In order to derive a computationally feasible analogue of the  multiscale finite element discretization~\eqref{eq:WANAL},
we start from a typically very large, maybe computationally inaccessible finite element space $\cS$ 
associated with a very strong refinement $\cT$ of $\cT^{(k)}$ that resolves  all  fine scale features of the multiscale
interface problem as necessary to provide the desired accuracy of order $h_k$.
Proceeding literally as  above with $\cH$ replaced by $\cS$, we obtain discrete versions of 
Theorems~\ref{theo:CDISCRET}, \ref{theo:CAPPROX}, and \ref{theo:RICHARDSON},
 where the iteration \eqref{eq:RICHARDSON} takes the form of a damped block Jacobi iteration. 

\section{Iterative subspace correction} \label{sec:SUBSPAC}
We now consider the construction and convergence analysis of  
subspace correction methods for  the fractal interface problem~\eqref{eq:FRACP} together with
computationally feasible discrete versions for $k$-scale finite element approximations~\eqref{eq:FED}.
Their convergence rates neither depend on the scales $k\in \N$ nor on the meshsize $h_k$.

The starting point is the two-level splitting 
\begin{equation} \label{eq:TWOSCALE}
\cH = \cV_0 + \sum_{G \in \Omega^{(k)}} \cV_G
\end{equation}
with 
\[
\cV_0= \cS_{\ell}, \qquad \cV_G = \{ v|_G\;|\; v \in \cH \}, \quad G \in \Omega^{(k)}
\]
and $1 \leq \ell < k$. In particular, each $v \in \cH$ can be decomposed into its local components
\[
v_\ell = \Pi_{\ell} v \in \cS_{\ell}, \qquad v_G = (v- \Pi_{\ell} v)|_G \in \cV_G, \quad G \in \Omega^{(k)}.
\]
Utilizing stability and approximation properties of $\Pi_{\ell}:\cH \to \cS_{\ell}$,
stability  and boundedness of the splitting \eqref{eq:TWOSCALE} 
with corresponding  constants $K_1$ and $K_2$ 
follows by similar arguments as in the proof of Lemma~\ref{lem:SPLITTSTAB}.
Therefore, the corresponding preconditioner 
\[
T = T_0+ \sum_{G \in \Omega^{(k)}}P_G
\]
with Ritz projections $P_0: \cH \to \cV_0$ and $P_G : \cH \to \cV_G$, $G \in \Omega^{(k)}$, respectively,
admits the bound $\kappa \leq K_1 K_2$ of the condition number $\kappa$ of $T: \cH \to \cH$.
This property directly entails corresponding bounds for the convergence rates of  preconditioned 
linear and nonlinear iterative schemes like Richardson or conjugate gradient methods.

In order to describe a sequential subspace correction method induced by the splitting \eqref{eq:TWOSCALE},
we introduce a numbering $\{G_1,\dots, G_m\}= \Omega^{(k)}$ of the cells 
and of the corresponding subspaces $\cV_i= \cV_{G_i}$ and Ritz projections $P_i= P_{\cV_i}$, $i=1,\dots,m$.  
We now consider the linear iteration
\begin{equation} \label{eq:MGCONT}
w_{0}= u^{{(\nu)}}, \quad w_{i+1} = w_i + P_{m-i} (u - w_i),\; i=0,\dots, m, \quad  u^{(\nu+1)} = w_{m+1},  
\end{equation}
for $\nu=0,1,\dots$ with arbitrary given iterate $u^{(0)}\in \cH$. 
Instead of boundedness \eqref{eq:K2}, convergence of \eqref{eq:MGCONT} 
relies on the following Cauchy-Schwarz-type inequality.
\begin{lemma}
Assume that for all $k \in \N$ and each $G$ in $\Omega^{(k)}$ the number of neighboring cells of $G$ from $\Omega^{(k)}$ is uniformly bounded by $c_N \in \R$.
Then the Cauchy-Schwarz-type inequality 
\[
\sum_{i, j=0}^m a( v_i , w_j)  \leq K_3 \left( \sum_{i =0}^m a( v_i , v_i) \right)^{1/2}  \left( \sum_{j=0}^m a( w_j , w_j) \right)^{1/2} 
\]
holds for all $v_i \in \cV_i$, $w_j \in \cV_j$, $i,j = 0, \dots, m$,
with a constant $K_3$ depending only on $c_N$.
\end{lemma}
\begin{proof}
For some fixed $G \in \Omega^{(k)}$, we introduce the local scalar product
\[
\begin{array}{c}
\displaystyle a_G( v,w) = \int_{G\backslash \Gamma}A\nabla v\cdot\nabla w\; dx +
\tfrac{1}{2}\sum_{j=1}^{k}\left(1+\fc\right)^{j}C_{j}\int_{\Gamma_{j} \cap \partial G}B\jump v\jump w\; d\Gamma_j \\
\displaystyle  \hspace{5cm}+ \sum_{j=k+1}^{\infty}\left(1+\fc\right)^{j}C_{j}\int_{\Gamma_{j} \cap G} B\jump v\jump w\; d\Gamma_j,
\qquad v, w \in \cH,
\end{array}
\]
with the property
\begin{equation} \label{eq:LOCSCA}
\sum_{G \in \Omega^{(k)}}a_G(v,w) = a(v,w), \qquad v, w \in \cH.
\end{equation}
As the common support of $v_i\in \cV_i$ and $w_j\in \cV_j$ is contained in $\overline{G}_i \cap\overline{G}_j$ for $i,j=1,\dots,m$,
the Cauchy-Schwarz inequality and Gershgorin's theorem lead to  
\[
\sum_{i, j =0}^m a_G( v_i, w_j)
\leq (c_G+1) \left( \sum_{i=0}^m a_G(v_i,v_i) \right)^{1/2}  \left(  \sum_{j=0}^m a_G(w_j,w_j) \right)^{1/2}
\]
with $c_G$ denoting the number of neighboring cells of $G$ from $\Omega^{(k)}$.
After summation over $G \in \Omega^{(k)}$,  the Cauchy-Schwarz inequality in $\R^{m+1}$ 
together with  \eqref{eq:LOCSCA} complete the proof.
\end{proof}

The following convergence result is based on the error propagation
\begin{equation} \label{eq:ERRORMG}
u-u^{(\nu+1)}= (I-P_0) \cdots (I-P_m)(u- u^{{(\nu)}}).
\end{equation}
Its proof can be taken literally, e.g.,  from~\cite[Theorem 5.2]{kornhuber2016numerical}.
\begin{theorem} \label{theo:MGCNV}
Assume that for all $k \in \N$ and each $G$ in $\Omega^{(k)}$ 
the number of neighboring cells of $G$ from $\Omega^{(k)}$ is uniformly bounded by $c_N \in \R$.
Then the iterative scheme \eqref{eq:MGCONT} is convergent with respect to the energy norm, and
\[
\aNorm{u - u^{(\nu +1)}}\leq \left(1- \tfrac{1}{K_1K_3^2} \right)\aNorm{u - u^{{(\nu)}}}
\]
holds for any initial iterate $u^{(0)} \in \cH$ with $K_1$, $K_3$ 
depending only on
the constants appearing  in Theorems~\ref{theo:UNIAPPROXALL}, \ref{theo:UNISTABALL},
the geometric constant $C_0$ in \eqref{eq:RKSUP}, $c_N$ and  the ellipticity constants $\fa$, $\fA$ from \eqref{eq:ELLIP}.
\end{theorem}

We emphasize that the two-level iteration \eqref{eq:MGCONT} is just a simple illustrative example for a subspace correction method
that can be analyzed using the projection operators suggested in Section~\ref{sec:PROJECTIONS}.
More efficient methods can be constructed in a similar way.
For example, a symmetric variant of \eqref{eq:MGCONT} that can be accelerated by conjugate gradients,
is obtained by augmenting each iteration step by additional corrections   $P_i(u - w_{m+1+ i})$
taken in reverse order $i=1,\dots,m$.
For detailed investigations, we refer to~\cite{podlesny20}.

The linear iteration  \eqref{eq:MGCONT} takes place  in  $\cH$ and thus 
requires the successive evaluation of Ritz projections $P_i$ to
infinite dimensional subspaces $\cV_i\subset \cH$, $i=1,\dots,m$. 
However,  replacing $\cH$ by a finite element space $\cS_K$ with some $K \geq k> \ell$,
the above considerations and convergence results literally translate 
to corresponding subspace correction methods 
for the finite element discretization  \eqref{eq:FED} with respect to $\cS_K$.
In particular, the discrete analogue of  \eqref{eq:MGCONT}
leads to a two-grid iteration with block Gau\ss -Seidel smoother on the fine grid $\cT^{(k)}$ that is globally converging 
with  convergence rate independent of the level $K$ and corresponding meshsize~$h_K$ of the discrete solution space $\cS_K$.

\section{Numerical Experiments}
In our two numerical experiments, we consider the finite element discretization  \eqref{eq:FED}
of the fractal interface problem \eqref{eq:FRACP} with $\Omega= (0,1)^2 \subset \R^2$,
$\fc = 1$, the identity matrix $A=I \in \R^{d \times d}$, $B=1$ and two different kinds of fractal interface networks.

In order to illustrate the theoretical findings of Section~\ref{sec:SUBSPAC},
we consider the discrete analogue of  the linear iteration \eqref{eq:MGCONT} in function space,
i.e., the two-grid method with block Gau\ss -Seidel smoother  
as induced by the two-level splitting \eqref{eq:TWOSCALE},
with coarse space  $\cS_\ell = \cS_1$. The fine grid level $k=K$ is selected to
coincide with the level of the underlying discrete solution space $\cS_K$, $K= 1, \dots, K_{\max}$.
We always use the initial iterate $u^{(0)} = u_{\cS_1}$, i.e., the finite element approximation on the coarse grid $\cT^{(1)}$.

In light of the hierarchical lower bound
\[
\HNorm{u _{\cS_{K+1}} - u_{\cS_K}}\leq \HNorm{u - u_{\cS_K}}
\]
 of the discretization error, the algebraic error is reduced up to discretization accuracy
 once the computationally feasible criterion
 \begin{equation} \label{eq:STOCRI}
\HNorm{u _{\cS_{K}} - u_{\cS_K}^{{(\nu)}}} \leq \HNorm{u_{S_{K+1}} - u_{\cS_K}}
\end{equation}
is fulfilled. We will use \eqref{eq:STOCRI} to determine the minimal number of iteration steps
as required to reduce the algebraic error below  discretization accuracy.

%
%

\subsection{Highly localized  interface network}
In our first numerical experiment, we consider the highly localized 
fractal interface network as depicted in Figure~\ref{fig:MODELNET}.
In this case, we have $d_k = \sqrt{2}\;4^{-k}$, $C_k=2^{k}$, and $r_k= 2^{1-k}$.
Hence, conditions \eqref{eq:DKZERO}, \eqref{eq:RKSUP} hold true and
the conditions \eqref{eq:SMALLCL},  \eqref{eq:QDEC} are satisfied for $\fc=1$.

Starting with the triangulation $\cT^{(1)}$ as obtained by two uniform  regular refinements of
the partition $\cT^{(0)}$ consisting of two congruent triangles, the triangulation
$\cT^{(k)}$ results from two uniform regular refinement steps applied to $\cT^{(k-1)}$ for $k=2, 3, \dots$.
We have $h_k = \sqrt{2}\;4^{-k}$ so that \eqref{eq:HKDK}  holds with $\delta = 1$.
For all $k \in \N$ and each $G$ in $\Omega^{(k)}$, the number of neighboring cells 
of $G$ from $\Omega^{(k)}$ is uniformly bounded by $c_N =6$.
As a consequence, the conditions for  uniform stability and approximation property of the projections $\Pi_k$, $k\in \N$, 
as stated in Theorem~\ref{theo:UNIAPPROXALL} and Theorem~\ref{theo:UNISTABALL}, respectively,
and for the uniform convergence result in Theorem~\ref{theo:MGCNV}
are satisfied in this case.

Table~\ref{tab:LOCNET} displays the error reduction factors  
\[
\rho_K^{(\nu)}=\frac{\| u_{\cS_K} - u_{\cS_K}^{(\nu)}\|}{\|u_{\cS_K} - u_{\cS_K}^{(\nu-1)}\|}, \qquad \nu =1, \dots, 9,
\]
together with their geometric mean $\rho_K $ for the levels $K=1,\dots, K_{\max}= 5$.
We observe that the error reduction factors  nicely converge to the convergence rates on each level $K$
and appear to saturate at $0.266$ with increasing  $K$.
According to the criterion \eqref{eq:STOCRI} the  discretization accuracy is  already reached after $3$ steps.

\begin{table}[h]  \label{tab:LOCNET}
\centering
\begin{tabular}{ccccc}
        $\nu$ & $K=2$ & $K=3$ & $K=4$ & $K=5$\\
        \hline
        $1$ & $0.208$ & $0.247$ & $0.252$ & $0.252$\\
        $2$ & $0.221$ & $0.259$ & $0.263$ & $0.263$\\
        $3$ & $0.223$ & $0.261$ & $0.265$ & $0.265$\\
        $4$ & $0.224$ & $0.261$ & $0.266$ & $0.266$\\
        $5$ & $0.224$ & $0.261$ & $0.266$ & $0.266$\\
        $6$ & $0.224$ & $0.261$ & $0.266$ & $0.266$\\
        $7$ & $0.224$ & $0.261$ & $0.266$ & $0.266$\\
        $8$ & $0.224$ & $0.261$ & $0.266$ & $0.266$\\
        $9$ & $0.224$ & $0.261$ & $0.266$ & $0.266$\\
        \hline
        $\rho_K$ & $0.222$ & $0.259$ & $0.264$ & $0.264$\\
        \hline\\
\end{tabular}
\caption{Highly localized interface network: \\Error reduction factors and geometric mean $\rho_K$ 
of two-level subspace correction method} \label{tab:LOCNET}
\end{table}

\subsection{Geologically inspired interface network}
In our second numerical experiment, we consider an interface network mimicking a fractal crystalline structure. 
The triangulation $\cT^{(1)}$ is obtained by four uniform regular refinement steps applied  to
the partition $\cT^{(0)}$ consisting of two congruent triangles, and the triangulation
$\cT^{(k+1)}$ results from uniform regular refinement of $\cT^{(k)}$ for $k=1, 2, \dots$. 
The level-k interfaces are inductively constructed as follows.\\
Let $G_0 = \Omega$ denote the initial cell with center $c = (0.5, 0.5)^T$ and  midpoints $l,t,r,b \in \R^2$ 
of its left, top, right, and bottom boundary. 
The level-$1$ interface $\Gamma_1$, as shown in the left picture of Figure~\ref{fig:CRYSTALNET}, 
then consists of four connected paths of edges in $\cE^{(1)}$ starting with $l,t,r,b$ and ending with $c$. 
These four paths must not self-intersect and must meet in and only in $c$.
With these constraints, the  actual selection of edges is made randomly
 with strong bias towards the straight line connecting the corresponding start and end points. 
 Once $\Gamma^{(1)} = \Gamma_1$ is constructed, 
 centers $c_i= (c_{i,1},c_{i,2})^T$ of the four resulting cells $G_i \in \Omega^{(1)}$, $i=1,\dots, 4$, are determined 
 in a similar way as described above. Each cell $G_i\in \Omega^{(1)}$ is either refined now or never. 
 The decision about refinement or $G_i \in \Omega^{(1)}_{\infty}$ is made randomly 
 according to the probability ${\mathcal P}(\min\{c_{i,1}, c_{i,2}\})$ with density $\rho(\xi)= 2(1-\xi)$, $\xi \in (0,1)$,
 i.e., with a linear bias towards the left and  the lower boundary of $\Omega$.
 In case of refinement, $G_i$ is split into four subcells by four paths of edges in $\cE^{(2)}$
 starting with midpoints of its left, top, right, and bottom boundary  and ending with $c_i$
 in analogy to the splitting of the initial cell $G_0$.
 The union of  all these paths constitutes  the level-2 interface $\Gamma_2$. 
 This procedure is repeated inductively to construct the interface networks $\Gamma_k$, $k=2, \dots,6$
(see Figure~\ref{fig:CRYSTALNET}).

Apparently, the resulting interface network does not satisfy the locality condition  \eqref{eq:QDEC}
and the other conditions stated in  Theorems~\ref{theo:UNIAPPROXALL}, \ref{theo:UNISTABALL}
that are finally sufficient for   the convergence result in Theorem~\ref{theo:MGCNV}  
are also unclear.

\begin{figure}[h]
\includegraphics[width=3cm]{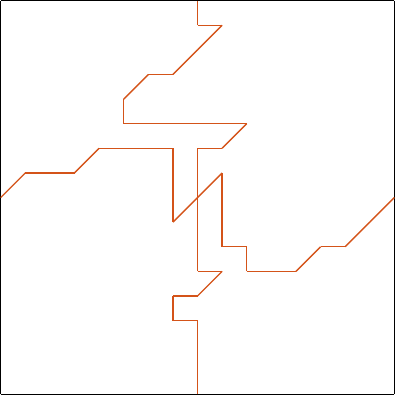}$\quad$
\includegraphics[width=3cm]{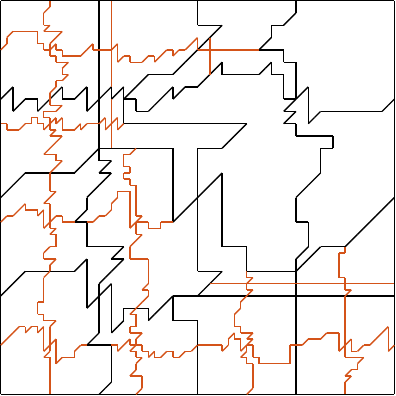}$\quad$
\includegraphics[width=3cm]{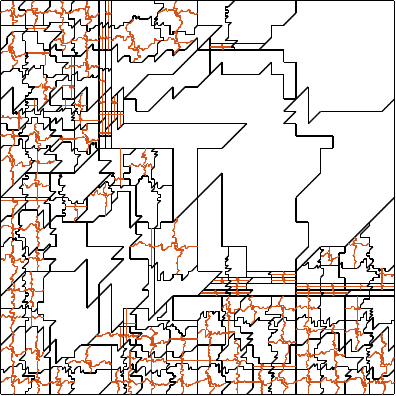}$\quad$
\includegraphics[width=3cm]{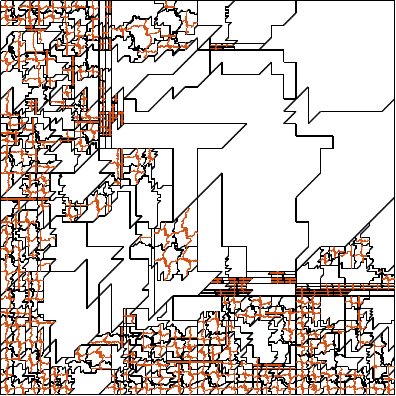}
\caption{\label{fig:CRYSTALNET} 
Geologically inspired interface network in $d=2$ space dimensions: 
 $\Gamma^{(1)}=\Gamma_1$ (red) and $\Gamma^{(k)}$ with $\Gamma_k$ (red) for $k=3,5,6$. }
\end{figure}

\begin{table}[h]  
\centering
\begin{tabular}{cccccc}
        $\nu$ & $K=2$ & $K=3$ & $K=4$ & $K=5$ & $K=6$\\
        \hline
        $1$ & $0.624$ & $0.696$ & $0.732$ & $0.744$ & $0.748$\\
        $2$ & $0.675$ & $0.735$ & $0.766$ & $0.775$ & $0.777$\\
        $3$ & $0.711$ & $0.758$ & $0.781$ & $0.788$ & $0.790$\\
        $4$ & $0.733$ & $0.773$ & $0.791$ & $0.796$ & $0.798$\\
        $5$ & $0.746$ & $0.785$ & $0.798$ & $0.803$ & $0.804$\\
        $6$ & $0.753$ & $0.792$ & $0.804$ & $0.808$ & $0.809$\\
        $7$ & $0.758$ & $0.798$ & $0.809$ & $0.812$ & $0.813$\\
        $8$ & $0.761$ & $0.802$ & $0.813$ & $0.816$ & $0.816$\\
        $9$ & $0.763$ & $0.805$ & $0.816$ & $0.818$ & $0.819$\\
        \hline
        $\rho_K$ & $0.723$ & $0.771$ & $0.790$ & $0.795$ & $0.797$\\
        \hline\\
\end{tabular}

\caption{Geologically inspired interface network:\\ 
Error reduction factors  and geometric mean $\rho_K$ of two-level subspace correction method} \label{tab:GEONET}
\end{table}

Nevertheless, the  error reduction factors as displayed Table~\ref{tab:GEONET} only moderately deteriorate in comparison with the highly localized case and even seem to saturate with increasing level $K$.
According to the criterion \eqref{eq:STOCRI} the  discretization accuracy is  already reached after $5$ steps.

\bibliographystyle{plain}
\bibliography{paper}

\begin{thebibliography}{10}

\bibitem{abdulle2012heterogeneous}
Assyr Abdulle, E~Weinan, Bj{\"o}rn Engquist, and Eric Vanden-Eijnden.
\newblock The heterogeneous multiscale method.
\newblock {\em Acta Numerica}, 21:1--87, 2012.

\bibitem{allaire1992homogenization}
Gr{\'e}goire Allaire.
\newblock Homogenization and two-scale convergence.
\newblock {\em SIAM Journal on Mathematical Analysis}, 23(6):1482--1518, 1992.

\bibitem{allaire1996multiscale}
Gr{\'e}goire Allaire and Marc Briane.
\newblock Multiscale convergence and reiterated homogenisation.
\newblock {\em Proceedings of the Royal Society of Edinburgh Section A:
  Mathematics}, 126(2):297--342, 1996.

\bibitem{baffico2008homogenization}
L{\'e}onardo Baffico, C{\'e}line Grandmont, Yvon Maday, and Axel Osses.
\newblock Homogenization of elastic media with gaseous inclusions.
\newblock {\em Multiscale Modeling \& Simulation}, 7(1):432--465, 2008.

\bibitem{bank1983some}
Randolph~E.\ Bank, Andrew~H.\ Sherman, and Alan Weiser.
\newblock Some refinement algorithms and data structures for regular local mesh
  refinement.
\newblock {\em Scientific Computing, Applications of Mathematics and Computing
  to the Physical Sciences}, 1:3--17, 1983.

\bibitem{bey2000simplicial}
J{\"u}rgen Bey.
\newblock Simplicial grid refinement: on {F}reudenthal's algorithm and the
  optimal number of congruence classes.
\newblock {\em Numerische Mathematik}, 85(1):1--29, 2000.

\bibitem{brenner1994two}
Susanne~C Brenner.
\newblock Two-level additive schwarz preconditioners for nonconforming finite
  elements.
\newblock {\em Contemporary Mathematics}, 180:9--9, 1994.

\bibitem{carstensen2006clement}
Carsten Carstensen.
\newblock Cl{\'e}ment interpolation and its role in adaptive finite element
  error control.
\newblock In {\em Partial differential equations and functional analysis},
  pages 27--43. Springer, 2006.

\bibitem{cazeaux2015homogenization}
Paul Cazeaux, C{\'e}line Grandmont, and Yvon Maday.
\newblock Homogenization of a model for the propagation of sound in the lungs.
\newblock {\em Multiscale Modeling \& Simulation}, 13(1):43--71, 2015.

\bibitem{cioranescu2013homogenization}
Doina Cioranescu, Alain Damlamian, and Julia Orlik.
\newblock Homogenization via unfolding in periodic elasticity with contact on
  closed and open cracks.
\newblock {\em Asymptotic Analysis}, 82(3-4):201--232, 2013.

\bibitem{clement1975approximation}
Philippe Cl{\'e}ment.
\newblock Approximation by finite element functions using local regularization.
\newblock {\em Revue fran{\c{c}}aise d'automatique, informatique, recherche
  op{\'e}rationnelle. Analyse num{\'e}rique}, 9(R2):77--84, 1975.

\bibitem{donato2004homogenization}
Patrizia Donato and Sara Monsurro.
\newblock Homogenization of two heat conductors with an interfacial contact
  resistance.
\newblock {\em Analysis and Applications}, 2(03):247--273, 2004.

\bibitem{dupont1980polynomial}
Todd Dupont and Ridgway Scott.
\newblock Polynomial approximation of functions in {S}obolev spaces.
\newblock {\em Mathematics of Computation}, 34(150):441--463, 1980.

\bibitem{efendiev2009multiscale}
Yalchin Efendiev and Thomas~Y Hou.
\newblock {\em Multiscale finite element methods: theory and applications},
  volume~4.
\newblock Springer Science \& Business Media, 2009.

\bibitem{ern2017finite}
Alexandre Ern and Jean-Luc Guermond.
\newblock Finite element quasi-interpolation and best approximation.
\newblock {\em ESAIM: Mathematical Modelling and Numerical Analysis},
  51(4):1367--1385, 2017.

\bibitem{evans2015measure}
Lawrence~Craig Evans and Ronald~F Gariepy.
\newblock {\em Measure theory and fine properties of functions}.
\newblock CRC press, 2015.

\bibitem{gao2014strength}
Xiang Gao and Kelin Wang.
\newblock Strength of stick-slip and creeping subduction megathrusts from heat
  flow observations.
\newblock {\em Science}, 345(6200):1038--1041, 2014.

\bibitem{grebenkov2006mathematical}
Denis~S Grebenkov, Marcel Filoche, and Bernard Sapoval.
\newblock Mathematical basis for a general theory of laplacian transport
  towards irregular interfaces.
\newblock {\em Physical Review E}, 73(2):021103, 2006.

\bibitem{gruais2017heat}
Isabelle Gruais and Dan Poli{\v{s}}evski.
\newblock Heat transfer models for two-component media with interfacial jump.
\newblock {\em Applicable Analysis}, 96(2):247--260, 2017.

\bibitem{hackbusch1997composite}
Wolfgang Hackbusch and Stefan~A Sauter.
\newblock Composite finite elements for the approximation of pdes on domains
  with complicated micro-structures.
\newblock {\em Numerische Mathematik}, 75(4):447--472, 1997.

\bibitem{heida2012stochastic}
Martin Heida.
\newblock Stochastic homogenization of heat transfer in polycrystals with
  nonlinear contact conductivities.
\newblock {\em Applicable Analysis}, 91(7):1243--1264, 2012.

\bibitem{heida2017fractal}
Martin Heida, Ralf Kornhuber, and Joscha Podlesny.
\newblock Fractal homogenization of multiscale interface problems.
\newblock {\em Multiscale Modeling \& Simulation}, 18(1):294--314, 2020.

\bibitem{hou1997multiscale}
Thomas~Y Hou and Xiao-Hui Wu.
\newblock A multiscale finite element method for elliptic problems in composite
  materials and porous media.
\newblock {\em Journal of computational physics}, 134(1):169--189, 1997.

\bibitem{hughes1998variational}
Thomas~JR Hughes, Gonzalo~R Feij{\'o}o, Luca Mazzei, and Jean-Baptiste Quincy.
\newblock The variational multiscale methodÑa paradigm for computational
  mechanics.
\newblock {\em Computer methods in applied mechanics and engineering},
  166(1-2):3--24, 1998.

\bibitem{Hummel1999}
H.K. Hummel.
\newblock {\em Homogenization of Periodic and Random Multidimensional
  Microstructures}.
\newblock PhD thesis, Technische Universit\"at Bergakademie Freiberg, 1999.

\bibitem{JKO1994}
Vasilii~Vasil'evich Jikov, Sergei~M.\ Kozlov, and Olga~Arsen'evna Ole{\u\i}nik.
\newblock {\em Homogenization of Differential Operators and Integral
  Functionals}.
\newblock Springer-Verlag, Berlin, 1994.

\bibitem{kornhuber2018analysis}
Ralf Kornhuber, Daniel Peterseim, and Harry Yserentant.
\newblock An analysis of a class of variational multiscale methods based on
  subspace decomposition.
\newblock {\em Mathematics of Computation}, 87(314):2765--2774, 2018.

\bibitem{kornhuber1994multilevel}
Ralf Kornhuber and Harry Yserentant.
\newblock Multilevel methods for elliptic problems on domains not resolved by
  the coarse grid.
\newblock {\em Contemporary Mathematics}, 180:49--49, 1994.

\bibitem{kornhuber2016numerical}
Ralf Kornhuber and Harry Yserentant.
\newblock Numerical homogenization of elliptic multiscale problems by subspace
  decomposition.
\newblock {\em Multiscale Modeling \& Simulation}, 14(3):1017--1036, 2016.

\bibitem{lancia2002transmission}
Maria~Rosaria Lancia.
\newblock A transmission problem with a fractal interface.
\newblock {\em Zeitschrift f{\"u}r Analysis und ihre Anwendungen},
  21(1):113--133, 2002.

\bibitem{maalqvist2014localization}
Axel M{\aa}lqvist and Daniel Peterseim.
\newblock Localization of elliptic multiscale problems.
\newblock {\em Mathematics of Computation}, 83(290):2583--2603, 2014.

\bibitem{mosco2015layered}
Umberto Mosco and Maria~Agostina Vivaldi.
\newblock Layered fractal fibers and potentials.
\newblock {\em Journal de Math{\'e}matiques Pures et Appliqu{\'e}es},
  103(5):1198--1227, 2015.

\bibitem{nagahama1994scaling}
Hiroyuki Nagahama and Kyoko Yoshii.
\newblock Scaling laws of fragmentation.
\newblock In {\em Fractals and Dynamic Systems in Geoscience}, pages 25--36.
  Springer, 1994.

\bibitem{oncken2012strain}
Onno Oncken, David Boutelier, Georg Dresen, and Kerstin Schemmann.
\newblock Strain accumulation controls failure of a plate boundary zone:
  Linking deformation of the central andes and lithosphere mechanics.
\newblock {\em Geochemistry, Geophysics, Geosystems}, 13(12), 2012.

\bibitem{oswald1993bpx}
Peter Oswald.
\newblock On a bpx-preconditioner for p1 elements.
\newblock {\em Computing}, 51(2):125--133, 1993.

\bibitem{peterseim2016variational}
Daniel Peterseim.
\newblock Variational multiscale stabilization and the exponential decay of
  fine-scale correctors.
\newblock In {\em Building bridges: connections and challenges in modern
  approaches to numerical partial differential equations}, pages 343--369.
  Springer, 2016.

\bibitem{pipping2016efficient}
Elias Pipping, Ralf Kornhuber, Matthias Rosenau, and Onno Oncken.
\newblock On the efficient and reliable numerical solution of rate-and-state
  friction problems.
\newblock {\em Geophysical Journal International}, 204(3):1858--1866, 2016.

\bibitem{podlesny20}
Joscha Podlesny.
\newblock {\em Multiscale Modelling and Simulation of Deformation Accumulation
  in Fault Networks}.
\newblock PhD thesis, Freie Universit\"at Berlin, 2020.

\bibitem{preusser20113d}
Tobias Preusser, Martin Rumpf, Stefan Sauter, and Lars~Ole Schwen.
\newblock 3d composite finite elements for elliptic boundary value problems
  with discontinuous coefficients.
\newblock {\em SIAM Journal on Scientific Computing}, 33(5):2115--2143, 2011.

\bibitem{rundle2003statistical}
John~B Rundle, Donald~L Turcotte, Robert Shcherbakov, William Klein, and
  Charles Sammis.
\newblock Statistical physics approach to understanding the multiscale dynamics
  of earthquake fault systems.
\newblock {\em Reviews of Geophysics}, 41(4), 2003.

\bibitem{sammis1986self}
Charles~G.\ Sammis, Robert~H.\ Osborne, J.~Lawford Anderson, Mavonwe Banerdt,
  and Patricia White.
\newblock Self-similar cataclasis in the formation of fault gouge.
\newblock {\em Pure and Applied Geophysics}, 124(1):53--78, 1986.

\bibitem{slobodeckii1958generalized}
L.N. Slobodecki{\i}.
\newblock Generalized {S}obolev spaces and their application to boundary
  problems for partial differential equations.
\newblock {\em Leningrad. Gos. Ped. Inst. Ucen. Zap}, 197:54--112, 1958.

\bibitem{triebel1982function}
Hans Triebel.
\newblock {\em Theorie of Function Spaces}.
\newblock Birh\"auser Basel, 1983.

\bibitem{turcotte1994crustal}
Donald~L.\ Turcotte.
\newblock Crustal deformation and fractals, a review.
\newblock In J\"orn~H.\ Kruhl, editor, {\em Fractals and Dynamic Systems in
  Geoscience}, pages 7--23. Springer, 1994.

\bibitem{turcotte1997fractals}
Donald~L. Turcotte.
\newblock {\em Fractals and Chaos in Geology and Geophysics}.
\newblock Cambridge University Press, 1997.

\bibitem{verfurth1999error}
R{\"u}diger Verf{\"u}rth.
\newblock Error estimates for some quasi-interpolation operators.
\newblock {\em ESAIM: Mathematical Modelling and Numerical Analysis},
  33(4):695--713, 1999.

\bibitem{weinan2003heterognous}
E~Weinan, Bjorn Engquist, et~al.
\newblock The heterognous multiscale methods.
\newblock {\em Communications in Mathematical Sciences}, 1(1):87--132, 2003.

\bibitem{xu1992iterative}
Jinchao Xu.
\newblock Iterative methods by space decomposition and subspace correction.
\newblock {\em SIAM review}, 34(4):581--613, 1992.

\bibitem{xu2008uniform}
Jinchao Xu and Yunrong Zhu.
\newblock Uniform convergent multigrid methods for elliptic problems with
  strongly discontinuous coefficients.
\newblock {\em Mathematical Models and Methods in Applied Sciences},
  18(01):77--105, 2008.

\bibitem{yserentant1993old}
Harry Yserentant.
\newblock Old and new convergence proofs for multigrid methods.
\newblock {\em Acta numerica}, 2:285--326, 1993.

\bibitem{zhikov2006homogenization}
Vasili~Vasil'evich Zhikov and Aleksandr~L.\ Pyatnitski{\u\i}.
\newblock Homogenization of random singular structures and random measures.
\newblock {\em Izvestiya Rossi\u\i skaya Akademiya Nauk. Seriya
  Matematicheskaya}, 70(1):23--74, 2006.

\end{thebibliography}
\end{document}